\documentclass[11pt,reqno]{amsart}
\usepackage{xifthen}
\usepackage{pkgfile}
\newtheorem{example}{Example}
\allowdisplaybreaks

\title[Monochromatic Subgraphs in Randomly Colored Graphons]{Monochromatic Subgraphs in Randomly Colored Graphons}

\author[Bhattacharya]{Bhaswar B. Bhattacharya}
\address{Department of Statistics, University of Pennsylvania, Philadelphia, USA,
{\tt bhaswar@wharton.upenn.edu}}

\author[Mukherjee]{Sumit Mukherjee\textsuperscript{*}}\thanks{\textsuperscript{*}Research partially supported by NSF grant DMS-1712037}
\address{Department of Statistics, Columbia University, New York, USA, {\tt  sm3949@columbia.edu}}

\begin{document}

\begin{abstract} Let $T(H, G_n)$ be the number of monochromatic copies of a fixed connected graph $H$ in a uniformly random coloring of the vertices of the graph $G_n$. In this paper we give a complete characterization of the limiting distribution of $T(H, G_n)$, when $\{G_n\}_{n \geq 1}$ is a converging sequence of dense graphs. When the number of colors grows to infinity, depending on whether the expected value remains bounded, $T(H, G_n)$ either converges to a finite linear combination of independent Poisson variables or a normal distribution. On the other hand, when the number of colors is fixed, $T(H, G_n)$ converges to a (possibly infinite) linear combination of independent centered chi-squared random variables. This generalizes the classical birthday problem, which involves understanding the asymptotics of $T(K_s, K_n)$, the number of monochromatic $s$-cliques in a  complete graph $K_n$ ($s$-matching birthdays among a group of $n$ friends), to general monochromatic subgraphs in a network. 
\end{abstract}


\maketitle

\section{Introduction}

Let $G_n$ be a simple labeled undirected graph with vertex set $V(G_n):=\{1,2,\cdots,|V(G_n)|\}$, edge set $E(G_n)$, and adjacency matrix $A(G_n)=\{a_{ij}(G_n),i,j\in  V(G_n)\}$. 
 In a {\it uniformly random $c_n$-coloring of $G_n$}, the vertices of $G_n$ are colored with $c_n$ colors as follows:  
\begin{equation}\P(v\in V(G_n) \text{ is colored with color } a\in \{1, 2, \ldots, c_n\})=\frac{1}{c_n},
\label{eq:uniform}
\end{equation}
independent from the other vertices. An edge $(a, b) \in E(G_n)$ is said to be {\it monochromatic} if $X_a=X_b$, where $X_v$ denotes the color of the vertex $v \in V(G_n)$ in a uniformly random $c_n$-coloring of $G_n$. Denote by  
\begin{align}\label{eq:me}
T(K_2, G_n)=\frac{1}{2} \sum_{1\le u \neq v\le |V(G_n)|} a_{uv}(G_n) \bm 1\{X_u=X_v\},
\end{align} 
the number of monochromatic edges in $G_n$. Note that $\P(T(K_{2}, G_n)>0)=1-\P(T(K_{2}, G_n)=0)=1-\chi_{G_n}(c_n)/c_n^{|V(G_n)|}$, where $\chi_{G_n}(c_n)$ counts the number of proper colorings of $G_n$ using $c_n$ colors. The function $\chi_{G_n}$ is known as the {\it chromatic polynomial} of $G_n$, and is a central object in graph theory \cite{chromaticbook,toft_book,toft_unsolved}. Moreover, the statistic \eqref{eq:me} shows up in various  applications, for example,  in the study of coincidences \cite{diaconismosteller} as a generalization of the birthday paradox \cite{barbourholstjanson,dasguptasurvey}, the Hamiltonian of the Ising/Potts models  \cite{bmpotts,BM}, and in non-parametric two-sample tests \cite{batu,fr}. This requires understanding the asymptotics of $T(K_2, G_n)$ for various graph sequences $G_n$. The limiting distribution of $T(K_{2}, G_n)$ has been recently characterized by Bhattacharya et al. \cite{BDM}, for any sequence of growing graphs $G_n$.

Given these results, it is natural to ask what happens for monochromatic triangles, and more general subgraphs. In this paper we consider the problem of determining the limiting distribution of the number of monochromatic copies of a general connected simple graph $H$, in a uniformly random $c_n$-coloring of a graph sequence $G_n$. Formally, this is defined as
$$T(H, G_n):=\frac{1}{|Aut(H)|}\sum_{\bm s \in V(G_n)_{|V(H)|}} \prod_{(a,b) \in E(H)}a_{s_a s_b}(G_n) \bm 1\{X_{=\bm s}\},$$
where:  
\begin{itemize}
\item[--] $ V(G_n)_{|V(H)|}$ is the set of all $|V(H)|$-tuples ${\bm s}=(s_1,\cdots, s_{|V(H)|})\in  V(G_n)^{|V(H)|}$ with distinct indices.\footnote{For a set $S$, the set $S^N$ denotes the $N$-fold cartesian product $S\times S \times \cdots \times S$.} Thus, the cardinality of $V(G_n)_{|V(H)|}$ is $\frac{|V(G_n)|!}{(|V(G_n)|-|V(H)|)!}$. 
\item[--] For any ${\bm s}=(s_1,\cdots, s_{|V(H)|}) \in  V(G_n)_{|V(H)|}$, 
\begin{align}\label{eq:sind}
\bm 1\{X_{=\bm s}\}:= \bm 1\{X_{s_1}=\cdots=X_{s_{|V(H)|}}\}.
\end{align}
\item[--] $Aut(H)$ is the {\it automorphism group} of $H$, that is, the number permutations $\sigma$ of the vertex set $V(H)$ such that $(x, y) \in E(H)$ if and only if $(\sigma(x), \sigma(y)) \in E(H)$. 
\end{itemize}


The class of possible limiting distributions of $T(H, G_n)$, for a general graph $H$, is extremely diverse \cite{BDM}, and, there appears to be no natural universality of the limiting distribution of $T(H, G_n)$, for a general graph sequence $\{G_n\}_{n \geq 1}$. In this paper, using results from the emerging theory of graph limits,  we provide a complete characterization of the limiting distribution of $T(H, G_n)$, for any simple connected graph $H$, whenever $\{G_n\}_{n\geq 1}$ is a convergent sequence of dense graphs \cite{lovasz_book}. Depending on the behavior of  $\E T(H, G_n)$ there are 3  different regimes:
\begin{itemize}
\item[(1)] $\E(T(H, G_n)) =O(1)$: In this case, $T(H, G_n)$ converges to a finite linear combination of independent Poisson random variables (Theorem \ref{thm:poisson}). 
\item[(2)] $\E(T(H, G_n)) \rightarrow \infty$, such that $c_n \rightarrow \infty$: Here, $T(H, G_n)$ is asymptotically Gaussian, after appropriate standardization (Theorem \ref{thm:clt}).
\item[(3)]$\E(T(H, G_n)) \rightarrow \infty$, such that $c_n =c$ is fixed: In this case, $T(H, G_n)$, after standardization, is asymptotically a (possibly infinite) linear combination of independent centered chi-squared random variables (Theorem \ref{thm:cfixed}).
\end{itemize}

We begin with a short background on graph limit theory. The results are formally stated in Section \ref{sec:results}.

\subsection{Graph Limit Theory}
\label{sec:graphlimit}

The theory of graph limits was developed by Lov\' asz and coauthors \cite{graph_limits_I,graph_limits_II,lovasz_book}, and has received phenomenal attention over the last few years. It builds a bridge between combinatorics and analysis, and has found applications in  several disciplines including statistical physics, probability, and statistics \cite{bmpotts,chatterjee_pd,CV}. 
For a detailed exposition of the theory of graph limits refer to Lov\' asz \cite{lovasz_book}. Here we mention the basic definitions about the convergence of graph sequences. If $F$ and $G$ are two graphs, then define the homomorphism density of $F$ into $G$ by
$$t(F,G) :=\frac{|\hom(F,G)|}{|V (G)|^{|V (F)|},}$$
where  $|\hom(F,G)|$ denotes the number of homomorphisms of $F$ into $G$. In fact, $t(F, G)$ is the proportion of maps $\phi: V (F) \rightarrow V (G)$ which define a graph homomorphism. Denote by $\hom_{\mathrm{inj}}(F, G)$ the number of injective maps from $F$ into $G$  which are homomorphisms, and $$t_{\mathrm{inj}}(F,G) :=\frac{|\hom_{\mathrm{inj}}(F,G)|}{|V(G)|(|V(G)|-1)\cdots (|V(G)|-|V(F)|+1)},$$
which is the proportion of injective maps which are homomorphisms. Moreover, denote by $t_{\mathrm{ind}}(F, G)$ the {\it induced homomorphism density}, that is, the proportion of injective maps $\phi: |V(F)| \rightarrow |V(G)|$, which satisfy $(\phi(x), \phi(y)) \in E(G)$ if and only if $(x, y) \in E(F)$:  
\begin{align}\label{eq:hmind}
t_{\mathrm{ind}}(F, G)=\frac{ \sum_{\bm s\in (V(G))_{|V(F)|}} \prod_{(a,b) \in E(F)}a_{s_a s_b}(G) \prod_{(a,b) \notin E(F)}(1-a_{s_a s_b}(G))}{|V(G)|(|V(G)|-1)\cdots (|V(G)|-|V(F)|+1)},
\end{align}
where $A(G)=(a_{ij}(G))_{i, j \in [|V(G)|]}$ is the adjacency matrix of $G$. 


To define the continuous analogue of graphs, consider $\sW$ to be the space of all measurable functions from $[0, 1]^2$ into $[0, 1]$ that satisfy $W(x, y) = W(y,x)$, for all $x, y$. For a simple graph $F$ with $V (F)= \{1, 2, \ldots, |V(F)|\}$, let
$$t(F,W) =\int_{[0,1]^{|V(F)|}}\prod_{(i,j)\in E(F)}W(x_i,x_j) \mathrm dx_1\mathrm dx_2\cdots \mathrm dx_{|V(F)|}.$$
\begin{defn}\cite{graph_limits_I,graph_limits_II,lovasz_book}\label{defn:graph_limit} A sequence of graphs $\{G_n\}_{n\geq 1}$ is said to {\it converge to $W$} if for every finite simple graph $F$, 
\begin{equation}
\lim_{n\rightarrow \infty}t(F, G_n) = t(F, W).
\label{eq:graph_limit}
\end{equation}
\end{defn}

If $G_n$ converges to $W$, according to definition above, then the injective homomorphsim densities converge: $t_{\mathrm{inj}}(F, G_n) \rightarrow t(F, W)$, for every simple graph $F$. Moreover, the induced homomorphism densities also converge, that is, $t_{\mathrm{ind}}(F, G_n) \rightarrow t_{\mathrm{ind}}(F, W)$, for every simple graph $F$, where 
\begin{align}\label{eq:indW}
t_{\mathrm{ind}}(F, W)=\int_{{[0, 1]}^{|V(F)|}} \prod_{(a,b) \in E(F)} W(x_a, x_b) \prod_{(a,b) \notin E(F)}(1-W(x_a, x_b)) \mathrm dx_1\mathrm dx_2 \cdots \mathrm dx_{|V(F)|}.
\end{align} 
The limit objects, that is, the elements of $\sW$, are called {\it graph limits} or {\it graphons}. A finite simple graph $G=(V(G), E(G))$ can also be represented as a graphon in a natural way: Define $f^G(x, y) =\boldsymbol 1\{(\ceil{|V(G)|x}, \ceil{|V(G)|y})\in E(G)\}$, that is, partition $[0, 1]^2$ into $|V(G)|^2$ squares of side length $1/|V(G)|$, and let $f^G(x, y)=1$ in the $(i, j)$-th square if $(i, j)\in E(G)$, and 0 otherwise. Observe that $t(F, f^{G}) = t(F,G)$ for every simple graph $H$ and therefore the constant sequence $G$ converges to the graph limit $f^G$. Define an equivalence relation on the space of graphons by $W_1\sim W_2$ iff $t(F,W_1)=t(F,W_2)$ for all simple graphs $F$. It turns out that the quotient space under this equivalence relation, equipped with the notion of convergence in terms of subgraph densities outlined above is a compact metric space using the {\it cut distance} (refer to \cite[Chapter 8]{lovasz_book}).

\subsection{Results}
\label{sec:results}

Throughout the paper, we will assume that $H$ is a finite, simple, and connected graph, and $G_n$ is a sequence of dense graphs converging to the graphon $W$ such that $t(H, W) >0$.  Depending on the limiting behavior of  $\E T(H, G_n)$ there are 3 different regimes.

\subsubsection{Linear Combination of Poissons}
\label{sec:thmlpoisson}

For a finite simple unlabeled graph $F$, let $N(F, G_n)$ be the number of copies of $F$ in $G_n$. Note that 
\begin{align}\label{eq:poissonexp}
N(H, G_n)=\frac{\hom_{\mathrm{inj}}(H, G_n)}{|Aut(H)|}\quad \text{and} \quad \E(T(H, G_n)) =\frac{N(H, G_n)}{c_n^{|V(H)|-1}}.
\end{align} 
We begin with the regime where the mean $\E(T(H, G_n))=O(1)$. In this case, the limit is a linear combination of independent Poisson variables, where the weights are determined by the limiting homomorphism densities of certain super-graphs of $H$. This is formalized in the following theorem:

\begin{thm}\label{thm:poisson} Let $G_n$ be a sequence of graphs converging to the graphon $W$, such that $t(H,W)>0$. Suppose $c_n\rightarrow\infty$, such that  $\E T(H, G_n) \rightarrow \lambda$. Then 
\begin{align}\label{eq:poissonlinear}
T(H, G_n) \dto \sum_{F \supseteq H: |V(F)|=|V(H)|} N(H, F) X_F,
\end{align}  
 where $X_F\sim \dPois\left(\lambda\cdot \frac{|Aut(H)|}{|Aut(F)|}\cdot \frac{t_{\mathrm{ind}}(F, W)}{t(H, W)}\right)$ and the collection $\{X_F: F \supseteq H \text{ and } |V(F)|=|V(H)|\}$ is independent.\footnote{For two graphs $F$ and $H$, $F \supseteq H$ means $F$ is a super-graph of $H$.} 
\end{thm}

The proof is based on a moment comparison technique, where the moments of $T(H, G_n)$ are compared with the moments of the corresponding random variable obtained by replacing every subset of $|V(H)|$ vertices with independent Bernoulli variables (refer to Section \ref{sec:lpoisson} for details).

\begin{remark} A useful special case of the above theorem, which generalizes the well-known birthday problem, is when $H=K_s$ is the $s$-clique (monochromatic cliques correspond to $s$-matching birthdays in a friendship network $G_n$). The asymptotics of multiple birthday matches have found many applications, for example, in the study of coincidences \cite[Problem 3]{diaconismosteller}, hash-function attacks in cryptology \cite{ns}, and the discrete logarithm problem \cite{BBBDL,ghdl}. Refer to Example \ref{ex:birthday} for more on the birthday paradox.  
\end{remark}

Arratia et al. \cite{arratia} used Stein's method based on dependency graphs to prove Poisson approximation theorems for $T(K_s, K_n)$, that is, the number of monochromatic $s$-cliques in a uniform coloring of a complete graph $K_n$ (see also Chatterjee et al. \cite{CDM}). Poisson limit theorems for $T(H, G_n)$ for general $H$ and arbitrary coloring distribution are given in Cerquetti and Fortini \cite{birthdayexchangeability}. They assumed that the distribution of colors
was exchangeable and proved that $T(H, G_n)$ converges in distribution to a mixture of Poisson. However, these results only give conditions under which the limit of $T(H, G_n)$ is a Poisson, and require several assumptions on the number of certain subgraphs in $G_n$ and the coloring distribution. On the other hand, Theorem \ref{thm:poisson} goes beyond the Poisson regime, and characterizes the limiting distribution of $T(H, G_n)$ for all dense graphs, under the uniform coloring distribution.

\subsubsection{Asymptotic Normality for Growing Colors}
\label{sec:thmclt}

Theorem \ref{thm:poisson} asserts that if $\E T(H, G_n) \rightarrow \lambda$, then the number of monochromatic copies of $H$ converges to a linear combination of Poissons. Recall that a Poisson random variable with mean growing to infinity converges to a standard normal distribution after centering by the mean and scaling by the standard deviation. Therefore, it is natural to wonder whether the same is true for $T(H, G_n)$, whenever $\E T(H, G_n) \rightarrow \infty$. To this end, define 
\begin{align}\label{eq:Z}
Z(H, G_n)=\frac{T(H, G_n)-\E T(H, G_n)}{\sqrt{\Var(T(H, G_n))}}.
\end{align} 
The theorem shows that $Z(H, G_n)$ has a universal CLT whenever  $\E T(H, G_n) \rightarrow \infty$ and $c_n \rightarrow \infty$. To this end, we define the {\it Wasserstein distance} between two probability measures $\mu$ and $\nu$ on $\R$ is, $$Wass(\mu, \nu) := \sup\left\{\int f \mathrm d \nu-\int f \mathrm d \mu: f: \R \rightarrow \R \text{ is 1-Lipschitz}\right\},$$ that is, supremum over all $f$ such that $|f(x)-f(y)| \leq |x-y|$. Moreover, for two nonnegative sequences $\{a_n\}_{n \geq 1}$  and $\{b_n\}_{n \geq 1}$,  $a_n \lesssim b_n$  means $a_n \leq C b_n$, for some constant $C >0$ and all $n$ large enough.

\begin{thm}\label{thm:clt} Let $G_n$ be a sequence of graphs converging to the graphon $W$, such that $t(H,W)>0$. If $c_n\rightarrow\infty$, then 
\begin{align}\label{eq:cltW}
Wass\left(Z(H, G_n), N(0,1)\right)\lesssim \left(\frac{c_n^{|V(H)|-1}}{|V(G_n)|^{|V(H)|}}\right)^{\frac{1}{2}}+ \left(\frac{1}{c_n}\right)^{\frac{1}{2}}.
\end{align}
This implies, if $t(H,W)>0$, then $Z(H, G_n) \dto N(0, 1)$, whenever $\E T(H, G_n) \rightarrow \infty$ and $c_n \rightarrow \infty$. 
\end{thm}

The proof of the above theorem is given in Section \ref{sec:normality}, and is based on a Stein's method based on dependency graphs.

\begin{remark}\label{rem:normal} For the case of monochromatic edges, \cite[Theorem 1.2]{BDM} showed that $Z(K_{2}, G_n) \dto N(0, 1)$, whenever $\E(T(K_2, G_n)) \rightarrow \infty$ such that $c_n \rightarrow \infty$, for {\it any} sequence of graphs $G_n$ with $|V(G_n)| \rightarrow \infty$. Error rates for the above CLT were obtained by Fang \cite{xiao}. The above theorem shows that this phenomenon extends to all simple connected graphs $H$, when $G_n$ is a converging sequence of dense graphs. Moreover, unlike in the case of edges, the density assumption $t(H, W)>0$ is, in general, necessary for $Z(H, G_n)$ to have a non-degenerate normal limit (see Example \ref{ex:normaldiamond}).  
\end{remark}

\subsubsection{Limiting Distribution for Fixed Number of Colors}
\label{sec:thmcfixed} 

In this section we derive the asymptotic distribution of the number of monochromatic subgraphs when  $\E T(H, G_n) \rightarrow \infty$  such that $c$ is fixed.
  
\begin{defn}({\it $2$-point homomorphism functions for graphons})\label{defn:tabxyHW}
Let $H$ be a labeled finite simple graph and $W$ is a graphon. Then, for $1 \leq u \ne v \leq |V(H)|$, the function $t_{u, v}(\cdot, \cdot, H, W): [0, 1]^2 \rightarrow [0, 1]$ is defined as:\footnote{For a graph $F=(V(F), E(F))$ and $S\subseteq V(F)$, the  {\it neighborhood} of $S$ in $F$ is $N_F(S)=\{v \in V(F): \exists ~u \in S \text{ such that } (u, v) \in E(F)\}$. Moreover, for $u, v \in V(F)$, $F\backslash \{u, v\}$ is the graph obtained by removing the vertices $u, v$ and all the edges incident on them.}  
\begin{align*}
& t_{u, v}(x, y, H, W) \\
&= W^+_{u,v}(x, y)\int_{[0,1]^{|V(H)|-2}} \prod_{r \in N_H(u)\backslash \{v\}}W(x, z_r) \prod_{s \in N_H(v) \backslash \{u\}} W(y, z_s) \prod_{(r, s) \in E(H\backslash \{u, v\}) } W(z_r, z_s)\prod_{r \notin \{u, v\}} \mathrm d z_r,
\end{align*}
with $W^+_{u,v}(x, y)=W(x, y)$ if $(u, v) \in E(H)$ and 1 otherwise. Note that $t_{v, u}(x, y, H, W)=t_{u, v}(y, x, H, W)$.

For example, when $H=K_{1, 2}$ is the $2$-star, with the central vertex labeled 1. Then the following hold:  
\begin{itemize} 
\item[--]$t_{1, 2}(x, y, K_{1, 2}, W) = t_{1, 3}(x, y, K_{1, 2}, W)= W(x, y) d_W(x)$, where $d_W(x)=\int_0^1 W(x, z)\mathrm d z$ is the degree function of the graphon $W$. 
\item[--] $t_{2, 3}(x, y, K_{1, 2}, W) = \int_{[0, 1]} W(x, z_1) W(y, z_1)  \mathrm dz_1$. 
\end{itemize}
Similarly, $t_{2, 1}(x, y, K_{1, 2}, W) =t_{3, 1}(x, y, K_{1, 2}, W)=W(x, y) d_W(y)$, and $$t_{3,2}(x, y, K_{1, 2}, W) = \int_{[0, 1]} W(x, z_1) W(y, z_1)  \mathrm dz_1.$$ More examples are computed in Section \ref{sec:examplecfixed}. 
\end{defn}

Using this definition we can now show that the limiting distribution of 
\begin{align}\label{eq:gamma}
\Gamma(H, G_n)=\frac{T(H,G_n)-\E T(H,G_n)}{|V(G_n)|^{|V(H)|-1}}, 
\end{align} 
is a linear combination of centered chi-squared random variables, whenever $\{G_n\}_{n \geq 1}$ converges and the number of colors is fixed. To this end, note that every bounded non-negative symmetric function $K:[0, 1]^2 \rightarrow \R$ defines an operator $T_K: L_2[0, 1]\rightarrow L_2(\R)$, by 
\begin{eqnarray*}
(T_Kf)(x)=\int_0^1K(x, y)f(y)\mathrm dy.
\end{eqnarray*}
$T_K$ is a Hilbert-Schmidt operator, which is compact and has a discrete spectrum, that is, a countable multi-set of non-zero real eigenvalues $\{\lambda_1(K), \lambda_2(K), \ldots, \}$, where every non-zero eigenvalue has finite multiplicity.

\begin{thm}\label{thm:cfixed} Let $G_n$ be a sequence of graphs converging  to the graphon $W$, such that $t(H, W)>0$. If $c_n=c$ is fixed, then 
\begin{align}\label{eq:gamma}
\Gamma(H, G_n)  \dto \frac{1}{c^{|V(H)|-1}} \cdot  \sum_{r=1}^\infty \lambda_r(H, W) \cdot \eta_r, 
\end{align} 
where 
\begin{itemize}
\item[--] $(\eta_1, \eta_2, \ldots)$ is a collection of independent $\chi^2_{(c-1)}-(c-1)$ random variables,\footnote{Recall that for a positive integer $a$, $\chi^2_{a}$ denotes the chi-squared distribution with $a$ degrees of freedom, that is, the sum of squares of $a$ independent standard normal ($N(0, 1)$) random variables.}

\item[--] $\{\lambda_1(H,W), \lambda_2(H,W), \ldots, \}$ is the multi-set of the non-zero eigenvalues of the bounded non-negative symmetric function
$W_H: [0, 1]^2 \rightarrow \R$ defined by 
\begin{align}\label{eq:WH}
W_H(x,y):=\frac{1}{2 |Aut(H)|}\sum_{1\leq u \ne v \leq |V(H)|} t_{u,v}(x,y,H,W).
\end{align}
(If the spectrum of $W_H$ is finite, then the sum in \eqref{eq:gamma} is interpreted by padding zeros to the spectrum.)
\end{itemize}
\end{thm}

Note that the function $W_H$ is symmetric because of the relation $t_{v, u}(x, y, H, W)=t_{u, v}(y, x, H, W)$ and is point-wise bounded by $W_H  \leq \frac{|V(H)|^2}{2 |Aut(H)|}$. The proof of the above theorem is given in Section \ref{sec:cfixed}. It has two main steps: 
\begin{itemize}
\item[--] The first step expands the random variable $\Gamma(H,G_n)$ as a polynomial in the i.i.d. {\it color vectors} $\{(\bm1\{X_v=a\})_{a\in [c]}: v\in V(G_n)\}$, and shows that only the quadratic term is dominant (refer to Lemmas \ref{lm:J2} and \ref{lm:delta2} for details).

\item[--] The second step shows that the limiting distribution of the quadratic term remains unchanged when the color vectors are replaced by a collection of i.i.d. Gaussian random vectors with the same mean and covariance structure (Lemma \ref{lm:invariance}). The result then follows by analyzing the asymptotics of the Gaussian counterpart. 
\end{itemize}

\subsection{Organization} The rest of the paper is organized as follows: The proof of Theorem \ref{thm:poisson} and its applications are given in Section \ref{sec:lpoisson}. The proof of Theorem \ref{thm:clt} is in Section \ref{sec:normality}. The proof of Theorem \ref{thm:cfixed} and related examples are discussed in Section \ref{sec:cfixed}.

\section{Linear Combination of Poissons: Proof of Theorem \ref{thm:poisson}}
\label{sec:lpoisson}

In this section we present the proof of Theorem \ref{thm:poisson} and discuss applications of this result in various examples. 

\subsection{Proof of Theorem \ref{thm:poisson}}

To analyze $T(H,G_n)$ we use the `independent approximation', where the indicators $\bm 1\{X_{=\bm s}\}$ are replaced by independent Bernoulli variables, for every subset of vertices in $G_n$ of size $|V(H)|$. To this end, define 
\begin{align}\label{eq:W}
J(H, G_n)=\frac{1}{|Aut(H)|}\sum_{\bm s \in V(G_n)_{|V(H)|}} \prod_{(a,b) \in E(H)}a_{s_a s_b}(G_n) J_{\bm s}, 
\end{align}
where  $\{J_{\bm s}, s_1 < s_2 < \cdots < s_{|V(H)|} \in V(G_n)\}$ is a collection of i.i.d. $\dBer(\frac{1}{c_n^{|V(H)|-1}})$ random variables, and if the coordinates of $\bm s$ are not in increasing order, define $J_{\bm s}=J_{\sigma(\bm s)}$, where $\sigma(\bm s)=(\sigma(s_1), \sigma(s_2), \cdots, \sigma(s_{|V(H)|}))$ such that $\sigma(s_1) <  \sigma(s_2) < \cdots < \sigma(s_{|V(H)|})$. This implies that  
\begin{align}\label{eq:Jcollection}
\left\{J_{\bm s}, \bm s \in {V(G_n) \choose {|V(H)|}}\right\}
\end{align} 
is a collection of i.i.d. $\dBer(\frac{1}{c_n^{|V(H)|-1}})$ random variables, where ${V(G_n) \choose {|V(H)|}}$ denotes the set of all $|V(H)|$-element subset of $V(G_n)$.

The following lemma shows that the moments of $T(H, G_n)$ and $J(H, G_n)$ are asymptotically close. Note that $A \lesssim_H B$, means $A\leq C(H) B$, where $C(H)$ is a constant that depends only on the graph $H$. Similarly,  $A \gtrsim_H B$, means $A \geq C(H) B$, where $C(H)$ is a constant that depends only on the graph $H$. 

\begin{lem}\label{lm:momentcompare} For any $r \geq 1$, 
\begin{align}\label{eq:TW}
\lim_{n \rightarrow \infty} \Big\{\E T(H, G_n)^r -\E J(H, G_n)^r\Big\}=0.
\end{align}
Moreover, there exists a constant $C=C(H, r)<\infty$ such that for all $n$ large, $\E T(H,G_n)^r\le C$ and  $\E J(H,G_n)^r\le C$. 
\end{lem}
\begin{proof}
 Let $\cA$ be the collection of all ordered $r$-tuples $\bm s_1, \bm s_2, \ldots, \bm s_r$, where $\bm s_j \in V(G_n)_{|V(H)|}$ for $j \in [r]$, with $\bm s_1=(s_{11}, s_{12}, \ldots, s_{1|V(H)|})$, $\bm s_2=(s_{21}, s_{22}, \ldots, s_{2|V(H)|}), \ldots, \bm s_r=(s_{r1}, s_{r2}, \ldots, s_{r|V(H)|})$, such  that $\prod_{(a,b) \in E(H)}a_{s_{ja} s_{jb}}(G_n)=1$, for every $j \in [r]$. Then by the multinomial expansion, 
\begin{align}
|\E T(H, G_n)^r -\E J(H, G_n)^r| & \leq  \frac{1}{|Aut(H)|^r}\sum_{\cA}  \Bigg|\E \prod_{t=1}^r  \bm 1\{X_{=\bm s_t}\}  - \E \prod_{t=1}^r  J_{\bm s_t} \Bigg| \nonumber \\
& = \frac{1}{|Aut(H)|^r} \sum_{\cA} \left|\frac{1}{c_n^{|V(F)|-\nu(F)}}- \frac{1}{c_n^{b|V(H)|-b}} \right|, \nonumber
\end{align}
where $F=F(\bm s_1,\cdots,\bm s_r)$ is the graph on vertex set $V(F)=\bigcup_{t=1}^r \bm s_t$ and edge set $\bigcup_{t=1}^r \{(s_{ta}, s_{tb}): (a,b) \in E(H)\}$, $\nu(F)$ is the number of connected components of $F$, and $b$ is the number of distinct $|V(H)|$-element subsets in the collection $\{\bm s_1, \bm s_2, \ldots, \bm s_r\}$.\footnote{Here, an ordered $|V(H)|$-tuple $\bm z=(z_1, z_2, \ldots, z_{|V(H)|}) \in V(G_n)_{|V(H)|}$ is considered as a $|V(H)|$-element subset $\{z_1, z_2, \ldots, z_{|V(H)|}\}$ of $V(G_n)$. For example, in the collection of ordered triples $\{(1,2,3), (3,2,1), (4,3,6), (5, 4,1), (3,6, 4), (5, 6, 8)\}$, there 4 distinct 3-element subsets, that is, $b=4$. Hereafter, we will slightly abuse notation and use the tuple and set interpretations interchangeably, whenever it is clear from the context.} Note that if the graph $F$ is connected, $|V(F)| -1 \leq  b|V(H)|-b$, and therefore, in general $|V(F)| -\nu(F) \leq  b|V(H)|-b$.

We now claim that $|V(F)|-\nu(F)<b|V(H)|-b$ implies $|V(F)|>|V(H)|\nu(F)$. Indeed, first note that trivially $|V(F)| \geq |V(H)|\nu(F)$. If $|V(F)| = |V(H)|\nu(F)$, then every connected component of $F$ is isomorphic $H$, that is, $\nu(F)=b$ and $|V(F)|-\nu(F)=b |V(H)|-b$, verifying the claim. Thus, setting $\cN_{p,q,r}$ to be the set of all $r$ ordered tuples $\bm s_1,\cdots,\bm s_r$ in $\cA$ such that $|\bigcup_{t=1}^r \bm s_t|=p$ and $\nu(F)=q$, we have
\begin{align}
 |\E T(H, G_n)^r -\E J(H, G_n)^r| & \lesssim_H \sum_{(p,q):q|V(H)|< p\le r|V(H)|}\sum_{\cN_{p,q,r}} \frac{1}{c_n^{p-q}} \nonumber \\ 
& \lesssim_H \sum_{(p,q):q|V(H)|< p\le r|V(H)|}\frac{|V(G_n)|^p}{c_n^{p-q}} \tag*{(using $|\cN_{p,q,r}|=O(|V(G_n)|^{p})$)}\nonumber \\
& \lesssim_H \sum_{(p,q):q|V(H)|< p\le r|V(H)|}  \frac{|V(G_n)|^{p}}{|V(G_n)|^{\frac{|V(H)|}{|V(H)|-1}(p-q)}}    \tag*{(since $c_n^{|V(H)|-1}=\Theta(|V(G_n)|^{|V(H)|})$)}\nonumber \\
& \lesssim_H  \sum_{(p,q):q|V(H)|< p\le r|V(H)|}  \frac{1 }{|V(G_n)|^{\frac{1}{|V(H)|-1}(p-q|V(H)|)}}. \nonumber 
\end{align} 
Since $p>q|V(H)|$, each term in the above sum converges to $0$, and because the sum is over a finite index set free of $n$, \eqref{eq:TW} follows. 

Finally, from the above arguments it also follows that 
$$\E T(H, G_n)^r   \lesssim_H \sum_{(p,q):q|V(H)|\le p\le r|V(H)|}   \frac{1}{|V(G_n)|^{\frac{1}{|V(H)|-1}(p-|V(H)|q)}}=O(1),$$
since $p\geq |V(H)|q$, for all $(p,q)$ in the above sum.
\end{proof}

Next, we show that the limiting distribution of $J(H, G_n)$ is a linear combination of independent Poisson random variables. 

\begin{lem}\label{lm:Wpoisson} Let $J(H, G_n)$ be as defined in \eqref{eq:W}.  Then 
$$J(H, G_n) \rightarrow \sum_{F \supseteq H: |V(F)|=|V(H)|} N(H, F) X_F,$$ 
in distribution and in moments, where $X_F\sim \dPois\left(\lambda\cdot \frac{|Aut(H)|}{|Aut(F)|}\cdot \frac{t_{\mathrm{ind}}(F, W)}{t(H, W)}\right)$ and the collection $\{X_F: F \supseteq H \text{ and } |V(F)|=|V(H)|\}$ is independent. 
\end{lem}

\begin{proof} Let ${V(G_n) \choose {|V(H)|}}$ be the collection of $|V(H)|$-element subsets of $V(G_n)$. For $S \subseteq V(G_n)$ denote by $G_n[S]$ the subgraph of $G_n$ induced on the set $S$. Then recalling the definition of $J(H, G_n)$ from \eqref{eq:W} gives 
\begin{align}\label{eq:Wsum}
J(H, G_n)&=\frac{1}{|Aut(H)|}\sum_{\bm s \in V(G_n)_{|V(H)|}} \prod_{(a,b) \in E(H)}a_{s_a s_b}(G_n) J_{\bm s} \nonumber \\
& = \sum_{\bm s \in {V(G_n) \choose {|V(H)|}}} N(H, G_n[\bm s]) J_{\bm s} \nonumber \\
&= \sum_{F \supseteq H: |V(F)|=|V(H)|} N(H, F) \sum_{\bm s \in {V(G_n) \choose {|V(H)|}}}  \bm 1\{ G_n[\bm s]=F\} \cdot J_{\bm s}.
\end{align}

Now, note that, by \eqref{eq:Jcollection}, the collection $$\left\{\sum_{\bm s \in {V(G_n) \choose {|V(H)|}}} \bm 1\{G_n[\bm s]=F\} \cdot J_{\bm s} :  F \supseteq H \text{ and } |V(F)|=|V(H)| \right\}$$  is independent, since for any two distinct super-graphs $F_1, F_2 \supseteq H$, with $|V(F_1)|=|V(F_2)|=|V(H)|$, the sets 
$$\left\{\bm s \in {V(G_n) \choose {|V(H)|}}: \bm 1\{G_n[\bm s]=F_1\} \right\}, \quad \text{and} \quad \left\{\bm s \in {V(G_n) \choose {|V(H)|}}: \bm 1\{G_n[\bm s]=F_2\} \right\}$$
are disjoint. Moreover, for every fixed $F$, $$J_n(F):=\sum_{\bm s \in {V(G_n) \choose {|V(H)|}} } \bm 1\{ G_n[\bm s]=F \} J_{\bm s}$$ is a sum of independent $\dBer\left(\frac{1}{c_n^{|V(H)|-1}}\right)$ random variables. Therefore, to prove theorem it suffices to show that $J_n(F) \rightarrow X_F$ (with $X_F$ as defined in the statement of the theorem) in distribution and in moments, which follows if we can prove that 
\begin{align}\label{eq:ZFexpI} 
\E(J_n(F)) \rightarrow \lambda \cdot \frac{|Aut(H)|}{|Aut(F)|}\cdot \frac{t_{\mathrm{ind}}(F, W)}{t(H, W)}.
\end{align} 

To show \eqref{eq:ZFexpI}, first note that $$\frac{|V(G_n)|^{|V(H)|}}{c_n^{|V(H)|-1}} = \frac{N(H, G_n)}{c_n^{|V(H)|-1}}\cdot \frac{|V(G_n)|^{|V(H)|}}{N(H, G_n)} =(1+o(1))\lambda  \cdot \frac{|V(G_n)|^{|V(H)|}}{\hom_{\mathrm{inj}}(H, G_n)/|Aut(H)|} \rightarrow \lambda \frac{|Aut(H)|}{t(H, W)},$$ 
using $\frac{\hom_{\mathrm{inj}}(H, G_n)}{|V(G_n)|^{|V(H)|}} = (1+o(1)) t_{\mathrm{inj}}(H, G_n) \rightarrow  t(H, W)$ (by \cite[(5.21)]{lovasz_book}.

Then recalling \eqref{eq:poissonexp}, 
\begin{align}
& \E\sum_{\bm s \in {V(G_n) \choose {|V(H)|}}}  \bm 1\{G_n[\bm s]=F\} \cdot J_{\bm s} \nonumber\\
& =\frac{1}{c_n^{|V(H)|-1}}  \sum_{\bm s \in{V(G_n) \choose {|V(H)|}}} \prod_{(a,b) \in E(F)}a_{s_a s_b}(G_n) \prod_{(a,b) \notin E(F)}(1-a_{s_a s_b}(G_n)) \nonumber \\
& =\frac{|V(G_n)|^{|V(H)|}}{c_n^{|V(H)|-1}} \frac{1}{|V(G_n)|^{|V(H)|}} \sum_{\bm s \in{V(G_n) \choose {|V(F)|}}} \prod_{(a,b) \in E(F)}a_{s_a s_b}(G_n) \prod_{(a,b) \notin E(F)}(1-a_{s_a s_b}(G_n)) \nonumber \\
& = (1+o(1)) \lambda \frac{|Aut(H)|}{t(H, W)} \cdot \frac{1}{|V(G_n)|^{|V(H)|}} \sum_{\bm s \in{V(G_n) \choose {|V(F)|}}} \prod_{(a,b) \in E(F)}a_{s_a s_b}(G_n) \prod_{(a,b) \notin E(F)}(1-a_{s_a s_b}(G_n)) \nonumber \\ 
& = (1+o(1)) \lambda \frac{|Aut(H)|}{t(H, W)} \cdot \frac{t_{\mathrm{ind}}(F, G_n)}{|Aut(F)|} \tag*{(recall \eqref{eq:hmind})}\nonumber \\
\label{eq:ZFexpII}& \rightarrow  \lambda \frac{|Aut(H)|}{t(H, W)} \cdot \frac{t_{\mathrm{ind}}(F, W)}{|Aut(F)|}. 
\end{align}   
where the last step uses $t_{\mathrm{ind}}(F, G_n) \rightarrow t_{\mathrm{ind}}(F, W)$ (since $G_n$ converges to $W$). \end{proof}

The proof of Theorem \ref{thm:poisson} can be easily completed using the above two lemmas:  
Let $Y:=\sum_{F \supseteq H: |V(F)|=|V(H)|} N(H, F) X_F$. By Lemma \ref{lm:momentcompare} and Lemma \ref{lm:Wpoisson}, $T(H, G_n)$ converges in moments to $Y$. Now, it is easy to check that $Y$ satisfies the Stieltjes moment condition \cite{Stieltjes}, therefore, it is uniquely determined by its moments. This implies $T(H, G_n) \dto Y$ as well, and hence completes the proof of Theorem \ref{thm:poisson}.

\subsection{Examples} Theorem \ref{thm:poisson} can be easily extended to converging sequence of dense random graphs, when the limits in \eqref{eq:graph_limit} hold in probability,  by conditioning on the graph, under the assumption that the graph and its coloring are jointly independent. Here, we compute the limiting distribution \eqref{eq:poissonlinear} for the Erd\H os-R\'enyi random graph.

\begin{example}(Erd\H os-R\'enyi random graphs) Let $G_n \sim G(n, p)$ be the Erd\H os-R\'enyi random graph. In this case $G_n$ converges to the constant function $W^{(p)}=p$. This implies that $t(H, W^{(p)})=p^{|E(H)|}$ and $t_{\mathrm{ind}}(F, W^{(p)})=p^{|E(F)|}(1-p)^{{|V(H)| \choose 2}-|E(F)|}$. Therefore, by Theorem \ref{thm:poisson},  choosing $c_n$ such that $\E T(H, G_n) =(1+o(1)) \frac{{|V(G_n)| \choose |V(H)|} p^{|E(H)|}}{c_n^{|V(H)|-1}} \rightarrow \lambda$, gives 
$$T(H, G_n) \dto \sum_{F \supseteq H: |V(F)|=|V(H)|} N(H, F) X_F,$$ 
where $X_F\sim \dPois\left(\lambda\cdot \frac{|Aut(H)|}{|Aut(F)|}p^{|E(F)|-|E(H)|}(1-p)^{{|V(H)| \choose 2}-|E(F)|}\right)$ and the collection $\{X_F: F \supseteq H \text{ and } |V(F)|=|V(H)|\}$ is independent. 
\begin{itemize}
\item[--]When $G_n=K_n$ (that is, $p=1$), $X_{K_{|V(H)|}}\sim \dPois\left(\lambda\cdot \frac{|Aut(H)|}{|V(H)|!}\right)=\dPois\left(\frac{\lambda}{N(H, K_{|V(H)|})}\right)$ and $X_F=0$ otherwise. Therefore, $$T(H, K_n) \dto N(H, K_{|V(H)|})\cdot \dPois\left(\frac{\lambda}{N(H, K_{|V(H)|})}\right).$$

\item[--] If $H=K_s$ is the complete graph, then $\{F \supseteq H: |V(F)|=|V(H)|\}=\{H\}$ and, therefore, $T(H, G_n) \dto \dPois(\lambda)$. 
\end{itemize}
\end{example}

When $H=K_s$ is the $s$-clique, we have a birthday problem on a general friendship network $G_n$.

\begin{example}\label{ex:birthday} (Birthday Problem) In the well-known birthday problem, $G_n$ is a friendship-network graph where the vertices are colored uniformly with $c_n=365$ colors (corresponding to birthdays). In this case, two friends will have the same birthday whenever the corresponding edge in the graph $G_n$ is monochromatic.  Therefore, $\P(T(K_{s}, G_n)>0)$ is the probability that there is an $s$-{\it fold birthday match}, that is, there are $s$ friends with the same birthday. For this problem, Theorem \ref{thm:poisson} can be used to do an approximate sample size calculation. For example, using $T(K_s, G_n) \dto \dPois(\lambda)$, where $\frac{N(K_s, G_n)}{c_n^{s-1}} \rightarrow \lambda$, and $\frac{1}{|V(G_n)|^s}N(K_s, G_n)\rightarrow \frac{1}{s!}t(K_s, W)$, gives 
$$\P(T(K_{s}, G_n)>0) \approx 1- e^{-\frac{N(K_s, G_n)}{c_n^{s-1}}} = p, \text{ which implies } |V(G_n)| \approx \left(\frac{s!}{t(K_s, W)} c_n^{s-1} \log\left(\frac{1}{1-p}\right) \right)^\frac{1}{s},$$
which approximates the minimum number of people needed to ensure a  $s$-fold birthday match in the network $G_n$, with probability at least $p$. When the underlying graph $G_n=K_n$ is the complete graph $K_n$ on $n$ vertices, this reduces to the classical birthday problem. For example, when $G_n=K_n$, $p=\frac{1}{2}$ and $s=3$, using $c_n=365$, the RHS above evaluates approximately to  82.1, that is, in any group of 83 people, with probability at least 50\%, there are three friends all having the same birthday. 
\end{example}

The assumption $t(H, W)>0$ in Theorem \ref{thm:poisson} enforces that $c_n^{|V(H)|-1}=\Theta(|V(G_n)|^{|V(H)|})$, and, in this regime, the limiting distribution of $T(H, G_n)$ is a linear combination of Poissons \eqref{eq:poissonlinear}. However when $t(H, W)=0$, for the scaling $c_n^{|V(H)|-1}\sim N(H, G_n)$ we can get `non-linear' limiting distributions, as shown below.

\begin{example}\label{ex:productpoisson}(Product of Independent Poissons) Let $G_n=K_{1, n, n}$, the complete 3-partite graph, with partite sets $\{z\}, B, C$ such that $|B|=|C|=n$. Note that every triangle in $G_n$ passes through $z$, hence, $N(K_3, G_n)=n^2$. In this case, the limiting graphon is $W(x,y)=\bm 1\{(x-\frac{1}{2})(y-\frac{1}{2}) \leq 0\}$, for which $t(K_3,W)=0$. However, if we color $G_n$ randomly with $c_n=n$ colors, such that $N(K_3, G_n)/c_n^2=1$, then $T(K_3, G_n)$ has a non-degenerate limiting distribution:  For $a \in [c_n]$, let $L_n(a)$ and $R_n(a)$ be the number of vertices in sets $B$ and $C$ with color $a$, respectively. Clearly, $$L_n(a) \sim \dBin(n, 1/c_n) \dto X, \quad R_n(a) \sim \dBin(n, 1/c_n) \dto Y,$$ where $X$ and $Y$ are independent $\dPois(1)$ variables. Thus, given the color of the vertex $z$ is $a$, 
$T(K_3, G_n)= L_n(a)  R_n(a) \dto X  Y,$ and therefore, unconditionally $T(K_3,G_n)\dto XY$, the product of two independent $\dPois(1)$ random variables.
\end{example}

\section{Asymptotic Normality: Proof of Theorem \ref{thm:clt}}
\label{sec:normality}

In this section we prove the asymptotic normality of $Z(H, G_n)$,  whenever $t(H,W)>0$ and $\E T(H, G_n) \rightarrow \infty$ such that $c_n \rightarrow \infty$. We begin the following definition.

\begin{defn}\label{defn:H2}Given a graph $H$ with vertices labeled $\{1, 2, \ldots, |V(H)|\}$ such that $a, b \in V(H)$, define the $(a, b)$-{\it join} of $H$, denoted by $H^2_{(a, b)}$ as follows: Let $H'$ be an isomorphic copy of $H$ with vertices $\{1', 2', \ldots, |V(H)|'\}$, where the vertex $s$ maps to the vertex $s'$, for $s \in V(H)$. The graph $H^2_{(a, b)}$ is obtained by identifying the vertex $a$ and $b$ in $H$ with the vertex $a'$ and $b'$ in $H'$, that is, $V(H^2_{(a, b)})=V(H)\bigcup V(H')\backslash\{a', b'\}$ and $$E(H^2_{(a, b)})=E(H)\bigcup E(H'\backslash\{a', b'\})\bigcup \{(a, x'): x' \in N_{H'}(a')\} \bigcup \{(b, y'): y' \in N_{H'}(b')\}.$$  
Note that $H^2_{(a, b)}$ has $2 |V(H)|-2$ vertices and $2|E(H)|-1$ or $2|E(H)|$ edges, depending on whether the edge $(a,b)$ is present or absent in $E(H)$ respectively. 
\end{defn}

\begin{lem}\label{lm:count} Let $H$ be a graph with vertices labeled $\{1, 2, \ldots, |V(H)|\}$ such that $a,b\in V(H)$. Then $t(H^2_{(a, b)},W)>0$, whenever $t(H,W)>0$.
\end{lem}

\begin{proof} Recalling Definition \ref{defn:tabxyHW}, gives  
\begin{align*}
t(H,W)=&\int_{[0,1]^2}  t_{a, b}(x_a, x_b, H, W)\mathrm dx_a\mathrm dx_b \\
\leq &  \left(\int_{[0,1]^2}  t_{a, b}^2(x_a, x_b, H, W) \mathrm dx_a\mathrm dx_b \right)^{\frac{1}{2}} =\left(t(H^2_{(a, b)},W)\right)^{\frac{1}{2}} \tag*{(by Cauchy-Schwarz)}\\
\end{align*}
This implies $t(H^2_{(a, b)},W) \geq t(H,W)^2 >0$.
\end{proof}

\begin{defn}\label{defn:Zs} For $\bm s \in V(G_n)_{|V(H)|}$, define 
\begin{align}\label{eq:Zs}
Z_{\bm s}:=\bm 1\{X_{=\bm s}\}-\E \bm 1\{X_{=\bm s}\}=\bm 1\{X_{=\bm s}\}-\frac{1}{c_n^{|V(H)|-1}},
\end{align}
where $\bm 1\{X_{=\bm s}\}$ is as defined in \eqref{eq:sind}. Then 
\begin{align}\label{eq:MGn}
T(H, G_n)-\E T(H, G_n)&=\sum_{{\bm s}\in V(G_n)_{|V(H)|}} M_{G_n}(\bm s, H)\left(\bm 1\{X_{=\bm s}\}-\frac{1}{c_n^{|V(H)|-1}}\right) \nonumber \\
&=\sum_{{\bm s}\in V(G_n)_{|V(H)|}} M_{G_n}(\bm s, H) Z_{\bm s}, 
\end{align}
where $ M_{G_n}(\bm s, H)=\frac{1}{|Aut(H)|} \prod_{(a,b) \in E(H)}a_{s_a s_b}(G_n)$. 
\end{defn}

The following lemma calculates the covariance of $Z_{\bm s}$ and $Z_{\bm t}$ and obtains a lower bound on the variance of $T(H, G_n)$. 

\begin{lem}\label{lm:var} The following hold: 
\begin{enumerate}
\item[(a)]
For $\bm s, \bm t \in V(G_n)_{|V(H)|}$, 
\begin{align*}
\E \bm 1\{X_{=\bm s}\}\bm 1\{X_{=\bm t}\}=\left\{
\begin{array}{ccc}
\frac{1}{c_n^{|{\bm s}\cup {\bm t}|-2}}  &  \text{ if } &  \left|{\bm s}\cup {\bm t}\right|=2|V(H)|  \\
\frac{1}{c_n^{|{\bm s}\cup {\bm t}|-1}}   &   \text{ if } &  \left|{\bm s}\cup {\bm t}\right| \leq 2|V(H)|-1.
\end{array}
\right.
\end{align*}

\item[(b)]
For $\bm s, \bm t \in V(G_n)_{|V(H)|}$, 
\begin{align*}
\E Z_{\bm s} Z_{\bm t}=\left\{
\begin{array}{ccc}
\frac{1}{c_n^{|{\bm s}\cup {\bm t}|-1}}-\frac{1}{c_n^{2 |V(H)|-2}}  &  \text{ if } &  \left|{\bm s}\cup {\bm t}\right| \leq 2|V(H)| -2  \\
0  &   \text{ if } &  \left|{\bm s}\cup {\bm t}\right| \in \{2|V(H)|-1, 2|V(H)|\}.
\end{array}
\right.
\end{align*}

\item[(c)]
If $t(H,W)>0$, then $$\Var(T(H, G_n)) \gtrsim  \max\left(\frac{|V(G_n)|^{|V(H)|}}{c_n^{|V(H)|-1}},\frac{|V(G_n)|^{2|V(H)|-2}}{c_n^{2|V(H)|-3}}\right).$$

\end{enumerate}

\end{lem}

\begin{proof}  If $\left|{\bm s}\cup {\bm t}\right|=2|V(H)|$, the indices ${\bm s},{\bm t}$ do not intersect. In this case, the expectation factorizes by independence, and $$\E \bm 1\{X_{=\bm s}\} \bm 1\{X_{=\bm t}\}=\E \bm 1\{X_{=\bm s}\}\E \bm 1\{X_{=\bm t}\}=\frac{1}{c_n^{2|V(H)|-2}}.$$
Otherwise, 
$$\E \bm 1\{X_{=\bm s}\}\bm 1\{X_{=\bm t}\}=\P(X_{s_1}=\cdots=X_{s_{|V(H)|}}=X_{t_1}=\cdots=X_{t_{|V(H)|}})=\frac{1}{c_n^{|{\bm s}\cup {\bm t}|-1}},$$
completing the proof of (a). The result in (b) follows from (a) and observing that $\E \bm 1\{X_{=\bm s}\}\E \bm 1\{X_{=\bm t}\}=\frac{1}{c_n^{2|V(H)|-2}}$.

To show (c) note that, by \eqref{eq:MGn}, 
\begin{align}\label{eq:varTHGn}
\Var&(T(H, G_n)) \nonumber \\ 
=&\sum_{{\bm s}\in V(G_n)_{|V(H)|}} \frac{M_{G_n}(\bm s, H)}{|Aut(H)|} \E Z_{\bm s}^2 + \sum_{{\bm s \ne \bm t}\in V(G_n)_{|V(H)|}} M_{G_n}(\bm s, H) M_{G_n}(\bm t, H) \E Z_{\bm s} Z_{\bm t}, 
\end{align}
using $ M_{G_n}(\bm s, H)^2=\frac{1}{|Aut(H)|^2} \prod_{(a,b) \in E(H)}a_{s_a s_b}(G_n)=\frac{1}{|Aut(H)|}M_{G_n}(\bm s, H)$. Now, since each of terms  in the covariance is non-negative by part (b), 
$$\Var(T(H, G_n)) \gtrsim_H \sum_{{\bm s}\in V(G_n)_{|V(H)|}} M_{G_n}(\bm s, H) \E Z_{\bm s}^2  \gtrsim_H  \frac{N(H,G_n)}{c_n^{|V(H)|-1}}\gtrsim \frac{|V(G_n)|^{|V(H)|}}{c_n^{|V(H)|-1}},$$ since $\frac{N(H,G_n)}{|V(G_n)|^{|V(H)|}}=(1+o(1))\frac{t_{\mathrm{inj}}(H,G_n)}{|Aut(H)|} \rightarrow \frac{t(H, W)}{|Aut(H)|} >0$, when $G_n$ converges to $W$. 

Another way to lower bound $\Var(T(H, G_n))$ is to use that the first term in \eqref{eq:varTHGn} is non-negative, and to consider only the sum over indices $|{\bm s}\cap {\bm t}|=2$ in the second term. The sum over such pairs ${\bm s},{\bm t} \in V(G_n)_{|V(H)|}$, that is, ${\bm s}=(s_1, s_2, s_3,\cdots,s_{|V(H)|}), {\bm t}=(s_1, s_2, t_3,\cdots, t_{|V(H)|})$ with the indices $\{s_1, s_2, s_3,\cdots,s_{|V(H)|}, $ $t_3,\cdots, t_{|V(H)|}\}$ all distinct gives 
$$\Var(T(H, G_n))\gtrsim_H \frac{N(H^2_{(1, 2)},G_n)}{c_n^{2|V(H)|-3}}\gtrsim \frac{|V(G_n)|^{2|V(H)|-2}}{c_n^{2|V(H)|-3}},$$ 
where the last step uses $\lim_{n \rightarrow \infty}\frac{1}{|V(G_n)|^{2|V(H)|-2}}N(H^2_{(1, 2)},G_n) \gtrsim t(H^2_{(1, 2)}, W) >0$ by Lemma \ref{lm:count}. Combining these two estimates give the desired lower bound on the variance. \end{proof}

In the following two lemmas we estimate the variance and covariance of the product of $Z_{\bm s}$ for 3 or 4 sets $\bm s \in V(G_n)_{|V(H)|}$, respectively. These will be used to control the error terms in Stein's method. 

\begin{lem}\label{lm:third} Let ${{\bm s}_1},{{\bm s}_2}, {{\bm s}_3}\in V(G_n)_{|V(H)|}$ be such that $\min\left\{|{{\bm s}_1}\bigcap {{\bm s}_2}|,|{{\bm s}_1}\bigcap {{\bm s}_3}|\right\}\ge 1$. Then the following hold:

\begin{enumerate}
\item[(a)]
If $|\bigcup_{a=1}^3 {\bm s}_a|=3|V(H)|-2$, then $\E (|Z_{{\bm s}_1}|Z_{{\bm s}_2}Z_{{\bm s}_3})=0.$

\item[(b)]
If $|\bigcup_{a=1}^3 {\bm s}_a|\le 3|V(H)|-2$, then $\E |Z_{{\bm s}_1}Z_{{\bm s}_2}Z_{{\bm s}_3}|\le \frac{8}{c_n^{|\bigcup_{a=1}^3 {\bm s}_a|-1}}$.

\end{enumerate}
\end{lem}

\begin{proof} We begin with the proof of case (a). First, we show that $\left|{{\bm s}_2}\bigcap \left({{\bm s}_1}\bigcup{{\bm s}_3}\right)\right|\le 1$. To see this, observe 
\begin{align*}
3|V(H)|-2=\left| \bigcup_{a=1}^3 {{\bm s}_a} \right|& =|V(H)|+\left|{{\bm s}_1}\bigcup {{\bm s}_3}\right|-\left|({{\bm s}_1}\bigcup {{\bm s}_3})\bigcap {{\bm s}_2}\right| \\
& \le |V(H)|+2|V(H)|-1-\left|({{\bm s}_1}\bigcup {{\bm s}_3})\bigcap {{\bm s}_2}\right|,
\end{align*}
which implies $\left|{{\bm s}_2}\bigcap \left({{\bm s}_1}\bigcup{{\bm s}_3}\right)\right|\le 1$. Therefore, $\left|{{\bm s}_2}\bigcap \left({{\bm s}_1}\bigcup{{\bm s}_3}\right)\right| = 1$, since $\min\left\{|{{\bm s}_1}\bigcap {{\bm s}_2}|,|{{\bm s}_1}\bigcap {{\bm s}_3}|\right\}\ge 1$. Let ${{\bm s}_2}\bigcap ({{\bm s}_1}\bigcup {{\bm s}_3})=\{j\}$, for some $j \in [|V(G_n)|]$. Then 
\begin{align*}
\E(|Z_{{\bm s}_1}|Z_{{\bm s}_2}Z_{{\bm s}_3})=\E\big(\E (Z_{{\bm s}_2}|X_j,\{X_i, i\in [|V(G_n)|]\backslash {\bm s}_2\} )|Z_{{\bm s}_1}|Z_{{\bm s}_3}\big)=0,
\end{align*}
since $\P(X_{j}=X_{q_2}\cdots=X_{q_{|V(H)|}}|X_j)=\frac{1}{c_n^{|V(H)|-1}}$, for any $q_2, q_3, \ldots q_{|V(H)|} \in V(G_n)\backslash \{j\}$ distinct. This completes the proof of (a).

Next, we prove case (b). By a direct expansion, it follows that 
\begin{align}\label{eq:b3} 
\E |Z_{{\bm s}_1}Z_{{\bm s}_2}Z_{{\bm s}_3}| & \le \E \bm 1\{X_{={\bm s}_1}\}\bm 1\{X_{={\bm s}_2}\}\bm 1\{X_{={\bm s}_3}\}  +\frac{4}{c_n^{3|V(H)|-3}} + \frac{\sum_{1\leq a< b \leq 3}\E \bm 1\{X_{={\bm s}_a}\}\bm 1\{X_{={\bm s}_b}\}}{c_n^{|V(H)|-1}} \nonumber \\
&=\frac{5}{c_n^{|\bigcup_{a=1}^3 {\bm s}_a|-1}} + \frac{\sum_{1\leq a< b \leq 3}\E \bm 1\{X_{={\bm s}_a}\}\bm 1\{X_{={\bm s}_b}\}}{c_n^{|V(H)|-1}},
\end{align}
since $|\bigcup_{a=1}^3 {\bm s}_a|\le 3|V(H)|-2$. To bound the second term in the RHS above, note that 
\begin{align}\label{eq:ub1}
\left|\bigcup_{a=1}^3 {\bm s}_a\right|=\left|{{\bm s}_1}\bigcup {{\bm s}_2}\right|+|V(H)|-\left|({{\bm s}_1}\bigcup {{\bm s}_2})\bigcap {{\bm s}_3} \right|\le \left|{{\bm s}_1}\bigcup {{\bm s}_2}\right|+|V(H)|-1,
\end{align} 
since $|{{\bm s}_3}\bigcap ({{\bm s}_1}\bigcup {{\bm s}_2})|\ge 1$. Then Lemma \ref{lm:var}(a) gives 
\begin{align}
\frac{1}{c_n^{|V(H)|-1}}\E \bm 1\{X_{={\bm s}_1}\}\bm 1\{X_{={\bm s}_2}\}&= \frac{1}{c_n^{|V(H)|-1}}\max\left(\frac{1}{c_n^{2|V(H)|-2}},\frac{1}{c_n^{|{{\bm s}_1}\bigcup {{\bm s}_2}|-1}}\right) \nonumber \\
& \le \max\left( \frac{1}{c_n^{3|V(H)|-3}}, \frac{1}{c_n^{|{{\bm s}_1}\bigcup {{\bm s}_2} \bigcup {{\bm s}_3}|}} \right) \tag*{(by \eqref{eq:ub1})} \nonumber \\
& \le \frac{1}{c_n^{|\bigcup_{a=1}^3 {\bm s}_a|-1}} \tag*{(using $|\bigcup_{a=1}^3 {\bm s}_a|\le 3|V(H)|-2$).}
\end{align}
Similar estimates holds for the pairs $({{\bm s}_1},{{\bm s}_3})$ and $({{\bm s}_2},{{\bm s}_3})$ as well, and, therefore, the RHS of \eqref{eq:b3} is bounded by $\frac{8}{c_n^{|\bigcup_{a=1}^3 {\bm s}_a|-1}}$.
\end{proof}

\begin{lem}\label{lm:cov}

Let ${{\bm s}_1},{{\bm s}_2},{{\bm s}_3},{{\bm s}_4}\in V(G_n)_{|V(H)|}$ be such that $\min\{|{{\bm s}_1}\bigcap {{\bm s}_2}|,|{{\bm s}_3}\bigcap {{\bm s}_4}|\}\ge 1$. Then the following hold: 
\begin{enumerate}
\item[(a)]
If $|\bigcup_{a=1}^4 {\bm s}_a|\in \{4|V(H)|-2,4|V(H)|-3\}$, then $\Cov(Z_{{\bm s}_1}Z_{{\bm s}_2},Z_{{\bm s}_3}Z_{{\bm s}_4})=0$.

\item[(b)]
If $|\bigcup_{a=1}^4 {\bm s}_a| \le 4 |V(H)| -4$, then $\Cov(Z_{{\bm s}_1}Z_{{\bm s}_2},Z_{{\bm s}_3}Z_{{\bm s}_4})\lesssim \frac{1}{c_n^{|\bigcup_{a=1}^4 {\bm s}_a|-1}}$.

\end{enumerate}

\end{lem}

\begin{proof} We begin with the proof of case (a). First, consider $|\bigcup_{a=1}^4 {\bm s}_a|=4|V(H)|-2$, then the index sets ${{\bm s}_1}\bigcup {{\bm s}_2}$ and ${{\bm s}_3}\bigcup {{\bm s}_4}$ are disjoint, and so, $\Cov(Z_{{\bm s}_1}Z_{{\bm s}_2},Z_{{\bm s}_3}Z_{{\bm s}_4})=0$. 

If $|\bigcup_{a=1}^4 {\bm s}_a|=4|V(H)|-3$ and $|({{\bm s}_1}\bigcup {{\bm s}_2})\bigcap ({{\bm s}_3}\bigcup {{\bm s}_4})| = 0$,  then the index sets ${{\bm s}_1}\bigcup {{\bm s}_2}$ and ${{\bm s}_3}\bigcup {{\bm s}_4}$ are disjoint, and $\Cov(Z_{{\bm s}_1}Z_{{\bm s}_2},Z_{{\bm s}_3}Z_{{\bm s}_4})=0$. Then the assumptions $|\bigcup_{a=1}^4 {\bm s}_a|=4|V(H)|-3$ and $\min\{|{{\bm s}_1}\bigcap {{\bm s}_2}|,|{{\bm s}_3}\bigcap {{\bm s}_4}|\}\ge 1$ implies, $|({{\bm s}_1}\bigcup {{\bm s}_2})\bigcap ({{\bm s}_3}\bigcup {{\bm s}_4})| = 1$. In this case, we must also have $|{{\bm s}_1}\bigcup {{\bm s}_2}|=|{{\bm s}_3}\bigcup {{\bm s}_4}|=2|V(H)|-1$, which implies 
\begin{align}\label{eq:e1}
\E Z_{{\bm s}_1}Z_{{\bm s}_2}=\E Z_{{\bm s}_3}Z_{{\bm s}_4}=0
\end{align} 
by Lemma \ref{lm:var}(a). Now, because $|({{\bm s}_1}\bigcup {{\bm s}_2})\bigcap ({{\bm s}_3}\bigcup {{\bm s}_4})|=1$, one of the sets ${{\bm s}_1 }\bigcap ({{\bm s}_3}\bigcup {{\bm s}_4})$ and ${{\bm s}_2 }\bigcap ({{\bm s}_3}\bigcup {{\bm s}_4})$ is non-empty.  Assuming, without loss of generality, that ${{\bm s}_1}\bigcap ({{\bm s}_3}\bigcup {{\bm s}_4})$ is non-empty, gives
 $$\left|{{\bm s}_1}\bigcup {{\bm s}_3}\bigcup {{\bm s}_4}\right|=|{{\bm s}_1}|+\left|{{\bm s}_3}\bigcup {{\bm s}_4}\right|-\left|{{\bm s}_1}\bigcap({{\bm s}_3}\bigcup {{\bm s}_4})\right|\leq |V(H)|+2|V(H)|-1-1=3|V(H)|-2,$$ 
 and $$1\leq \left|{{\bm s}_2}\bigcap ({{\bm s}_1}\bigcup {{\bm s}_3}\bigcup {{\bm s}_4})\right|=|V(H)| + \left|{{\bm s}_1}\bigcup {{\bm s}_3}\bigcup {{\bm s}_4}\right| -(4|V(H)|-3) \leq 1.$$ Denoting ${{\bm s}_2}\bigcap ({{\bm s}_1}\bigcup {{\bm s}_3}\bigcup {{\bm s}_4})=\{j\}$, then gives 
\begin{align}\label{eq:e2}
\E Z_{{\bm s}_1}Z_{{\bm s}_2}Z_{{\bm s}_3}Z_{{\bm s}_4}=\E \{ \E( Z_{{\bm s}_2}|X_j,\{X_i, i \in [|V(G_n)|]\backslash {{\bm s}_2}\}) Z_{{\bm s}_1}Z_{{\bm s}_3}Z_{{\bm s}_4}\}=0.
\end{align}
Combining \eqref{eq:e1} and \eqref{eq:e2}, $\Cov(Z_{{\bm s}_1}Z_{{\bm s}_2},Z_{{\bm s}_3}Z_{{\bm s}_4})=0$, whenever $|\bigcup_{a=1}^4 {\bm s}_a|=4|V(H)|-3$. This completes the proof of (a).

Next, we prove case (b). Without loss of generality, assume $|({{\bm s}_1}\bigcup {{\bm s}_2})\bigcap({{\bm s}_3}\bigcup {{\bm s}_4})|\ge 1$, because otherwise the covariance is $0$ to begin with. As in Lemma \ref{lm:third}, it suffices to find bounds up to fourth joint moments of $\bm 1\{X_{=\bm s}\}$. To this end, we first claim that 
\begin{align}\label{eq:3_III}
\frac{1}{c_n^{|V(H)|-1}}\E \bm 1\{X_{={\bm s}_1}\}\bm 1\{X_{={\bm s}_2}\}\bm 1\{X_{={\bm s}_3}\}& 
\leq \frac{1}{c_n^{ |\bigcup_{a=1}^4 {\bm s}_a|-1}},
\end{align}
and a similar bound holds for all other three triples. For proving \eqref{eq:3_III}, note that $|(\bigcup_{a=1}^3 {\bm s}_a)\bigcap {{\bm s}_4}|\ge |{{\bm s}_3}\bigcap {{\bm s}_4}|\ge 1$ which gives
\begin{align}\label{eq:sunion_I}
 \left|\bigcup_{a=1}^4 {\bm s}_a\right|=\left|\bigcup_{a=1}^3 {\bm s}_a\right|+|V(H)|-\left|\left(\bigcup_{a=1}^3 {\bm s}_a\right)\bigcap {{\bm s}_4}\right|\le \left|\bigcup_{a=1}^3 {\bm s}_a\right|+|V(H)|-1.
 \end{align} 
Now, there are two cases:  
\begin{itemize}

\item $|({{\bm s}_1}\bigcup {{\bm s}_2})\bigcap \bm s_3|\ge 1$: In this case, by \eqref{eq:sunion_I},  
\begin{align}\label{eq:3_N}
\frac{1}{c_n^{|V(H)|-1}}\E \bm 1\{X_{={\bm s}_1}\}\bm 1\{X_{={\bm s}_2}\}\bm 1\{X_{={\bm s}_3}\}=\frac{1}{c_n^{|\bigcup_{a=1}^3 {\bm s}_a|+|V(H)|-2}}\le  \frac{1}{c_n^{ |\bigcup_{a=1}^4 {\bm s}_a|-1}},
\end{align}
and so we have verified \eqref{eq:3_III}.

\item $|({{\bm s}_1}\bigcup {{\bm s}_2})\bigcap \bm s_3|=0$: In this case, 
\begin{align*}
\frac{1}{c_n^{|V(H)|-1}}\E \bm 1\{X_{={\bm s}_1}\}\bm 1\{X_{={\bm s}_2}\}\bm 1\{X_{={\bm s}_3}\}
& =\frac{1}{c_n^{|\bm s_1 \cup \bm s_2|+2|V(H)|-3}}\le  \frac{1}{c_n^{ |\bigcup_{a=1}^4 {\bm s}_a|-1}},
\end{align*}
which verifies \eqref{eq:3_III}. Note that in the last inequality we have used the bound
\begin{align*}
\left|\bigcup_{a=1}^4{\bm s}_a \right| =&\left|{\bm s}_1\bigcup {\bm s}_2 \right|+\left|{\bm s}_3\bigcup {\bm s}_4 \right|-\left|({\bm s}_1\bigcup {\bm s}_2)\bigcap ({\bm s}_3\bigcup {\bm s}_4) \right| \nonumber \\ 
& \le \left|{\bm s}_1\bigcup {\bm s}_2 \right|+2|V(H)|-2, 
\end{align*}
using $|{\bm s}_3\bigcup {\bm s}_4| \leq 2|V(H)|-1$, since $|{\bm s}_3\bigcap {\bm s}_4|\geq 1$, and $|({{\bm s}_1}\bigcup {{\bm s}_2})\bigcap({{\bm s}_3}\bigcup {{\bm s}_4})|\ge 1$. 
\end{itemize}

Next, proceeding to bound  expectations of two tuples, observe that $|{{\bm s}_3}\bigcup {{\bm s}_4}|- |({{\bm s}_1}\bigcup {{\bm s}_2})\bigcap ({{\bm s}_3}\bigcup {{\bm s}_4})|\le 2|V(H)|-2$, (since $|({{\bm s}_1}\bigcup {{\bm s}_2})\bigcap ({{\bm s}_3}\bigcup {{\bm s}_4})|\ge 1, |{{\bm s}_1}\bigcup {{\bm s}_2}|\le 2|V(H)|-1.$) 
This implies 
\begin{align*}
 \left|\bigcup_{a=1}^4 {\bm s}_a\right|=\left|{{\bm s}_1}\bigcup{{\bm s}_2}\right|+\left|{{\bm s}_3}\bigcup {{\bm s}_4}\right|-\left|({{\bm s}_1}\bigcup {{\bm s}_2})\bigcap ({{\bm s}_3}\bigcup {{\bm s}_4})\right|\le\left|{{\bm s}_1}\bigcup {{\bm s}_2}\right|+2|V(H)|-2,
\end{align*}
and so
\begin{align}\label{eq:2}
\frac{1}{c_n^{2|V(H)|-2}}\E \bm 1\{X_{={\bm s}_1}\}\bm 1\{X_{={\bm s}_2}\} & \le \frac{1}{c_n^{2|V(H)|-2}}\max\left(\frac{1}{c_n^{2|V(H)|-2}},\frac{1}{c_n^{|{{\bm s}_1}\bigcup {{\bm s}_2}|-1}}\right) \nonumber \\ 
& \le \max\left( \frac{1}{c_n^{4|V(H)|-4}}, \frac{1}{c_n^{|\bigcup_{a=1}^4 {\bm s}_a|-1}} \right),
\end{align}
 and a similar bound applies for all the other five pairs. Thus, expanding the fourth moment and using \eqref{eq:3_III} and \eqref{eq:2} gives
\begin{align}\label{eq:41}
\E |Z_{{\bm s}_1}Z_{{\bm s}_2}Z_{{\bm s}_3}Z_{{\bm s}_4}|
\lesssim \frac{1}{c_n^{|\bigcup_{a=1}^4 {\bm s}_a|-1}}+\frac{1}{c_n^{4|V(H)|-4}}\lesssim \frac{1}{c_n^{|\bigcup_{a=1}^4 {\bm s}_a|-1}}.
\end{align}
Finally, by Lemma \ref{lm:var}(b), if $\max(|{{\bm s}_1}\bigcup {{\bm s}_2}|,|{{\bm s}_3}\bigcup {{\bm s}_4}|)\ge 2|V(H)|-1$, then $\E Z_{{\bm s}_1}Z_{{\bm s}_2}\E Z_{{\bm s}_3}Z_{{\bm s}_4}=0$. Thus, assume that $|{{\bm s}_1}\bigcup {{\bm s}_2}|\le 2|V(H)|-2, |{{\bm s}_3}\bigcup {{\bm s}_4}|\le 2|V(H)|-2$, and so 
\begin{align}\label{eq:24}
|\E Z_{{\bm s}_1}Z_{{\bm s}_2} \E Z_{{\bm s}_3} Z_{{\bm s}_4}|\le \frac{1}{c_n^{|{{\bm s}_1}\bigcup {{\bm s}_2}|+{{\bm s}_3}\bigcup {{\bm s}_4}|-2}}\le \frac{1}{c_n^{|\bigcup_{a=1}^4 {\bm s}_a|-1}},
\end{align}
where the last inequality uses the fact that $|({{\bm s}_1}\bigcup {{\bm s}_2})\bigcap ({{\bm s}_3}\bigcup {{\bm s}_4})|\ge 1$. Combining \eqref{eq:41} along with \eqref{eq:24} completes the proof of the lemma. 
\end{proof}

\begin{proof}[Proof of Theorem \ref{thm:clt}] Recall $Z(H, G_n)$ from \eqref{eq:Z}. Note  that (recalling \eqref{eq:MGn}), 
$$Z(H, G_n)=\frac{\sum_{{\bm s}\in V(G_n)_{|V(H)|}} \prod_{(a,b) \in E(H)}a_{s_a s_b}(G_n)  Z_{\bm s}}{|Aut(H)|\sqrt{\Var(T(H, G_n))}}=\frac{\sum_{{\bm s}\in \cA_{H}^{(n)}}   Z_{\bm s}}{|Aut(H)|\sqrt{\Var(T(H, G_n))}},$$ where $\cA_{H}^{(n)}$ be the set all  ${{\bm s}}\in V(G_n)_{|V(H)|}$ such that $\prod_{(a,b) \in E(H)}a_{s_a s_b}(G_n)=1$. 
Now, for every ${{\bm s}_1}\in \cA_{H}^{(n)}$ let $$N_{{\bm s}_1}:=\{{{\bm s}_2} \in \cA_{H}^{(n)} : |{{\bm s}_1}\bigcap {{\bm s}_2}|\ge 1\}.$$ In other word, $N_{{\bm s}_1}$ is the subset of tuples in $\cA_{H}^{(n)}$ which have at least one index common with ${{\bm s}_1}$. Then by the Stein's method based on dependency graphs (see proof of Lemma 1 in \cite{chatterjee_dg}), we have 
\begin{align}
Wass\left(Z(H, G_n), N(0,1)\right) \lesssim_H R_1 + R_2,
\label{eq:chatterjee}
\end{align}
where 
$$R_1=\left( \Var\left(\frac{\sum_{{{\bm s}_1}\in \cA_{H}^{(n)}}Z_{{{\bm s}_1}}\sum_{{{\bm s}_2}\in N_{{\bm s}_1}}Z_{{\bm s}_2}}{\sigma_n^2}\right) \right)^{\frac{1}{2}}, ~ R_2=\frac{\sum_{{{\bm s}_1}\in \cA_{H}^{(n)}}\E|Z_{{\bm s}_1}|\left(\sum_{{{\bm s}_2}\in N_{{\bm s}_1}}Z_{{\bm s}_2}\right)^2}{\sigma_n^3},$$
with $\sigma_n^2=\Var(T(H, G_n))$. 

We will bound each of the terms above separately. To begin with, observe 
\begin{align*}
\Var \Bigg(\frac{1}{\sigma_n^2}&\sum_{{{\bm s}_1}\in \cA_{H}^{(n)}}Z_{{\bm s}_1}\sum_{{{\bm s}_2}\in N_{{\bm s}_1}}Z_{{\bm s}_2}\Bigg) \\
&=\frac{1}{\sigma_n^4}\sum_{{{\bm s}_1}\in \cA_{H}^{(n)}}\sum_{{{\bm s}_2}\in N_{{{\bm s}_1}}}\sum_{{{\bm s}_3}\in \cA_{H}^{(n)}}\sum_{{{\bm s}_4}\in N_{{{\bm s}_3}}}\Cov({Z}_{{{\bm s}_1}}{Z}_{{{\bm s}_2}},{Z}_{{{\bm s}_3}}{Z}_{{{\bm s}_4}}). 
\end{align*}
Let $\ell=|\bigcup_{a=1}^4 {\bm s}_a|$ and use Lemma \ref{lm:cov}(a) to conclude that the above covariance vanishes unless $\ell\le 4|V(H)|-4$. Thus, using Lemma \ref{lm:cov}(b), an upper bound to the RHS above is given by 
\begin{align*}
\frac{1}{\sigma_n^4}\sum_{\ell=|V(H)|}^{4|V(H)|-4}\frac{|V(G_n)|^\ell}{c_n^{\ell-1}} & \lesssim_H\frac{1}{\sigma_n^4}\left(\frac{|V(G_n)|^{|V(H)|}}{c_n^{|V(H)|-1}}+\frac{|V(G_n)|^{4|V(H)|-4}}{c_n^{4|V(H)|-5}}\right) \\
& \lesssim \frac{|V(G_n)|^{|V(H)|}}{c_n^{|V(H)|-1}}\times \frac{c_n^{2|V(H)|-2}}{|V(G_n)|^{2|V(H)|}}+ \frac{|V(G_n)|^{4|V(H)|-4}}{c_n^{4|V(H)|-5}}\times \frac{c_n^{4|V(H)|-6}}{|V(G_n)|^{4|V(H)|-4}} \\
&=\frac{c_n^{|V(H)|-1}}{|V(G_n)|^{|V(H)|}}+\frac{1}{c_n},
\end{align*}
where the last inequality uses Lemma \ref{lm:var}(c). Therefore
\begin{align}\label{eq:cha1}
R_1^2 \lesssim_H \Var\left(\frac{1}{\sigma_n^2}\sum_{{{\bm s}_1}\in  \cA_{H}^{(n)}}Z_{{\bm s}_1}\sum_{{{\bm s}_2}\in N_{{\bm s}_1}}Z_{{\bm s}_2}\right)
\lesssim \frac{c_n^{|V(H)|-1}}{|V(G_n)|^{|V(H)|}}+\frac{1}{c_n}.
\end{align}

Proceeding to bound $R_2$ in \eqref{eq:chatterjee}, gives
\begin{align*}
R_2=\frac{1}{\sigma_n^3} \sum_{{{\bm s}_1}\in \cA_{H}^{(n)}}\E |Z_{{\bm s}_1}|\left(\sum_{{{\bm s}_2}\in N_{{\bm s}_1}}Z_{{\bm s}_2}\right)^2=\frac{1}{\sigma_n^3} &\sum_{{{\bm s}_1}\in \cA_{H}^{(n)}}\sum_{{{\bm s}_2},{{\bm s}_3}\in N_{{\bm s}_1}} \E |Z_{{\bm s}_1}|Z_{{\bm s}_2}Z_{{\bm s}_3}
\end{align*}
Again let $\ell=|\bigcup_{a=1}^3 {\bm s}_a|$ and use Lemma \ref{lm:third}(a) to conclude that the above vanishes when $\ell =3|V(H)|-2$.  Thus, using the bound in  Lemma \ref{lm:third}(b), an upper bound to the RHS above is 
\begin{align*}
\frac{1}{\sigma_n^3}\sum_{\ell=|V(H)|}^{3|V(H)|-3}\frac{|V(G_n)|^\ell}{c_n^{\ell-1}}&\lesssim \frac{1}{\sigma_n^3}\left(\frac{|V(G_n)|^{|V(H)|}}{c_n^{|V(H)|-1}}+\frac{|V(G_n)|^{3|V(H)|-3}}{c_n^{3|V(H)|-4}}\right) \\& \lesssim \frac{|V(G_n)|^{|V(H)|}}{c_n^{|V(H)|-1}}\times \frac{c_n^{\frac{3}{2}(|V(H)|-1)}}{|V(G_n)|^{\frac{3}{2}|V(H)|}}+ \frac{|V(G_n)|^{3|V(H)|-3}}{c_n^{3|V(H)|-4}}\times \frac{c_n^{\frac{3}{2}(2|V(H)|-3)}}{|V(G_n)|^{3|V(H)|-3}} \\
&= \left(\frac{c_n^{|V(H)|-1}}{|V(G_n)|^{|V(H)|}}\right)^{\frac{1}{2}}+\left(\frac{1}{c_n}\right)^{\frac{1}{2}},
\end{align*}
where again the last inequality uses Lemma \ref{lm:var}(c). Therefore, 
\begin{align}\label{eq:cha2}
R_2=\frac{1}{\sigma_n^3} \sum_{{{\bm s}_1}\in V(G_n)_{|V(H)|}}\E |Z_{{\bm s}_1}|\left(\sum_{{{\bm s}_2}\in N_{{\bm s}_1}}Z_{{\bm s}_2}\right)^2\lesssim  \left(\frac{c_n^{|V(H)|-1}}{|V(G_n)|^{|V(H)|}}\right)^{\frac{1}{2}}+\left(\frac{1}{c_n}\right)^{\frac{1}{2}}.
\end{align}
Combining \eqref{eq:chatterjee} with \eqref{eq:cha1} and \eqref{eq:cha2} completes the proof of \eqref{eq:cltW}.

To see that the error term in the RHS of \eqref{eq:cltW} goes to zero, first note that $\E T(H, G_n)=\frac{1}{|Aut(H)|} \cdot \frac{\hom_{\mathrm{inj}}(H, G_n)}{c_n^{|V(H)|-1}}$ (recall \eqref{eq:poissonexp}). Then using  $\frac{1}{|V(G_n)|^{|V(H)|}}\hom_{\mathrm{inj}}(H, G_n) \rightarrow t(H, W)$, as $G_n$ converges to the graphon $W$, it follows that  the first term in the RHS of Theorem \ref{thm:clt} goes to zero whenever $\E T(H, G_n) \rightarrow \infty$. Therefore, $Z(H, G_n) \dto N(0, 1)$, whenever $\E T(H, G_n) \rightarrow \infty$ and $c_n \rightarrow \infty$. 
\end{proof}

As noted in Remark \ref{rem:normal}, the asymptotic normality of $Z(K_2, H)$ holds for all graph sequences $\{G_n\}_{n \geq 1}$, as long as $\E T(K_2, H) \rightarrow \infty$ and $c_n \rightarrow \infty$.  However, we cannot expect the normality of $Z(H, G_n)$, for general graphs $H$, to extend beyond dense graphs, without further assumptions, as shown below:
 
\begin{example}\label{ex:normaldiamond}For $n\geq 1$, denote by $\cD_n=(V(\cD_n), E(\cD_n))$ the $n$-{\it pyramid}: $V(\cD_n)=\{a, b, c_1, c_2, \ldots, c_n\}$ and $$E(\cD_n)=\{(a, b), (a, c_1), (a, c_2), \ldots (a, c_n), (b, c_1), (b, c_2), \ldots (b, c_n)\}.$$ In other words, the $n$-pyramid is the complete 3-partite graph $K_{1, 1, n}$. (Note that the 1-pyramid is the triangle ($\cD_1=K_3$) and the 2-pyramid $\cD_2$ is the 4-cycle with a diagonal.) Now, let $G_n$ be the disjoint union of the $n$-pyramid $\cD_n$ and the complete bipartite graph $K_{n, n}$ and $H=K_3$ the triangle. Choose $c_n \rightarrow \infty$ such that $c_n=o(\sqrt n)$. In this case, $\E T(K_3, G_n)=\frac{n}{c_n^2} \rightarrow \infty$, but $\lim_{n \rightarrow \infty} t(K_3, G_n)=0$, therefore, Theorem \ref{thm:clt} does not apply. Moreover, as every triangle in $G_n$  must pass through the base vertices $a, b$ of the pyramid $\cD_n$, 
$$\P(T(K_3,G_n)>0 ) \leq \P(X_a=X_b)=\frac{1}{c_n} \rightarrow 0.$$
This implies, $T(K_3, G_n) \pto 0$,  and hence, $Z(K_3, G_n) \pto 0$, that is, $Z(K_3, G_n)$ does not converge to a non-degenerate limiting distribution.  
\end{example}

\section{Limiting Distribution for fixed number of colors}
\label{sec:cfixed}

In this section we derive the limiting distribution for the number of monochromatic subgraphs when the number of colors is fixed. 
The proof of Theorem \ref{thm:cfixed} is given in Section \ref{sec:cfixedpf}. Examples are discussed in Section \ref{sec:examplecfixed}.

\subsection{Proof of Theorem \ref{thm:cfixed}}
\label{sec:cfixedpf}

We begin with the following observation: 

\begin{obs}\label{obs:sumprod} For $\bm s \in V(G_n)_{|V(H)|}$, let $Z_{\bm s}=\bm 1\{X_{=\bm s}\}-\frac{1}{c^{|V(H)|-1}}$. Then 
\begin{align}\label{eq:product_J}
Z_{\bm s}=\sum_{a=1}^c  \sum_{\substack{J \subseteq V(H) \\ |J| \geq 2}} \frac{1}{c^{|V(H)|-|J|}}\prod_{j \in J}\left(\bm1\{X_{s_j}=a\}-\frac{1}{c}\right).
\end{align}
\end{obs}

\begin{proof} To begin with note that 
$$\frac{1}{c^{|V(H)|-1}} \sum_{a=1}^c  \sum_{\substack{J \subseteq V(H) \\ |J| =1 }} \prod_{j \in J}\left(\bm1\{X_{s_j}=a\}-\frac{1}{c}\right)=\frac{1}{c^{|V(H)|-1}} \sum_{a=1}^c  \sum_{v=1}^{|V(H)|} \left(\bm1\{X_{v}=a\}-\frac{1}{c}\right)=0,$$
since, for every $v \in V(G_n)$, $\sum_{a=1}^c\left(\bm1\{X_{v}=a\}-\frac{1}{c}\right)=0$. Therefore, denoting the RHS of \eqref{eq:product_J} by $\tilde Z_{\bm s}$, we get 
\begin{align}\label{eq:prodJ_I}
\tilde Z_{\bm s}=\sum_{a=1}^c  \sum_{\substack{J \subseteq V(H)}} \frac{1}{c^{|V(H)|-|J|}}\prod_{j \in J}\left(\bm1\{X_{s_j}=a\}-\frac{1}{c}\right)-\frac{1}{c^{|V(H)|-1}},
\end{align}
where the $\frac{1}{c^{|V(H)|-1}}$ adjustment cancels with the term corresponding to taking $J=\emptyset$ in the sum. Here, we use the convention that $\prod_{j\in J}\left(\bm1\{X_{s_j}=a\}-\frac{1}{c}\right)=1$ if $J=\phi$. Now, expanding the product $\prod_{j \in J}\left(\bm1\{X_{s_j}=a\}-\frac{1}{c}\right)$ in  the RHS above gives, 
\begin{align}\label{eq:prodJ_II}
\sum_{\substack{J \subseteq V(H)}} \frac{1}{c^{|V(H)|-|J|}} \prod_{j \in J}\left(\bm1\{X_{s_j}=a\}-\frac{1}{c}\right) & = \sum_{\substack{J \subseteq V(H)}} \frac{1}{c^{|V(H)|-|J|}} \sum_{J' \subseteq J} \frac{(-1)^{|J|-|J'|}}{c^{|J|-|J'|}} \prod_{j \in J'} \bm1\{X_{s_j}=a\} \nonumber \\ 
& = \sum_{\substack{J \subseteq V(H)}} \sum_{J' \subseteq J} \frac{(-1)^{|J|-|J'|}}{c^{|V(H)|-|J'|}} \prod_{j \in J'} \bm1\{X_{s_j}=a\}\nonumber \\ 
& = \sum_{s=0}^{|V(H)|}  \sum_{\substack{J \subseteq V(H) \\ |J| = s }} \sum_{J' \subseteq J} \frac{(-1)^{s-|J'|}}{c^{|V(H)|-|J'|}} X_a(J'),
\end{align}
where $X_a(J')=\prod_{j \in J'} \bm1\{X_{s_j}=a\}$.  The RHS of \eqref{eq:prodJ_II} can be rewritten as the weighted sum over subsets $J' \subseteq V(H)$. Note that the coefficient of $X(J')$ in the above sum is 
$$\sum_{s=|J'|}^{|V(H)|} 
{|V(H)|-|J'| \choose s-|J'|}  \frac{(-1)^{s-|J'|}}{c^{|V(H)|-|J'|}} = \frac{1}{c^{|V(H)|-|J'|}}  \sum_{t=0}^{|V(H)|-|J'|} {|V(H)|-|J'| \choose t} (-1)^t=0,$$
whenever $|V(H)| \ne |J'|$. This follows by noting that there is ${|V(H)|-|J'| \choose s-|J'|}$ ways of completing the given set $J'$ to a $s$-element subset of $|V(H)|$. Therefore, the only term in \eqref{eq:prodJ_II} which is non-zero corresponds to taking $J' = V(H)$, and \eqref{eq:prodJ_I} simplifies to 
\begin{align*}
\tilde Z_{\bm s}=\sum_{a=1}^c X_a(V(H)) - \frac{1}{c^{|V(H)|-1}}&=\sum_{a=1}^c \prod_{j=1}^{|V(H)|} \bm 1\{X_{s_j}=a\} -\frac{1}{c^{|V(H)|-1}}  \nonumber \\ 
&=\bm 1\{X_{=\bm s}\} -\frac{1}{c^{|V(H)|-1}}= Z_{\bm s},
\end{align*}
as required.
 \end{proof}

Using this observation, $T(H, G_n)$ can be written as a polynomial in the i.i.d. color vectors $\{(\bm1\{X_v=a\})_{a\in [c]}: v\in V(G_n)\}$. 
\begin{align}\label{eq:Tcenter}
 T(H, G_n)-\E(T(H, G_n)) 
 &= \sum_{{\bm s}\in V(G_n)_{|V(H)|}}M_{G_n}(\bm s, H) Z_{\bm s} \nonumber \tag*{(recall \eqref{eq:MGn})}\\
& = \sum_{\bm s \in V(G_n)_{|V(H)|}} M_{G_n}(\bm s, H) \sum_{a=1}^c \sum_{\substack{J \subseteq V(H) \\ |J| \geq 2}} \frac{1}{c^{|V(H)|-|J|}}  \prod_{j \in J}\left(\bm1\{X_{s_j}=a\}-\frac{1}{c}\right) \nonumber \\
& = \sum_{\substack{J \subseteq V(H) \\ |J| \geq 2}}  T_J(H, G_n), 
\end{align} 
where 
\begin{align}\label{eq:TJHGn}
T_J(H, G_n):=\sum_{a=1}^c \sum_{\bm s \in V(G_n)_{|V(H)|}} M_{G_n}(\bm s, H)  \frac{1}{c^{|V(H)|-|J|}}  \prod_{j \in J}\left(\bm1\{X_{s_j}=a\}-\frac{1}{c}\right).
\end{align}

\begin{lem}\label{lm:J2}
For every $J \subseteq V(H)$ such that $|J| \geq 3$, $T_J(H, G_n)=o_P(|V(G_n)|^{|V(H)|-1})$.
\end{lem}

\begin{proof} Fix $J \subseteq V(H)$ such that $|J| \geq 3$. To begin with note that 
\begin{align}\label{eq:expQ}
\E T_J(H, G_n)=\sum_{a=1}^c \sum_{\bm s \in V(G_n)_{|V(H)|}} M_{G_n}(\bm s, H)  \frac{1}{c^{|V(H)|-|J|}}   \E\left(\prod_{j \in J}\bm1\{X_{s_j}=a\}-\frac{1}{c}\right) =0, 
\end{align}
since $$ \E \left(\prod_{j \in J}\bm1\{X_{s_j}=a\}-\frac{1}{c}\right)= \prod_{j\in J}\E\left(\bm1\{X_{s_j}=a\}-\frac{1}{c}\right)=0,$$ since $\{\left(\bm1\{X_{s_j}=a\}-\frac{1}{c}\right): j\in J)\}$ is a collection of independent random variables.  

The second moment of $T_J(H, G_n)$ equals
\begin{align}\label{eq:exp2QI}
\sum_{a, a' \in [c]} \sum_{\bm s, \bm s' \in V(G_n)_{|V(H)|}}    \frac{M_{G_n}(\bm s, H) M_{G_n}(\bm s', H)}{c^{2|V(H)|-2|J|}}  \E \prod_{j \in J}\left(\bm1\{X_{s_j}=a\}-\frac{1}{c}\right)\left(\bm1\{X_{s'_j}=a'\}-\frac{1}{c}\right).
\end{align}
Now, if  there exists $s_0 \in \{s_j: j \in J\} \backslash \{s_j': j \in J\}$, then 
\begin{align*}
&\E \prod_{j \in J}\left(\bm1\{X_{s_j}=a\}-\frac{1}{c}\right)\left(\bm1\{X_{s'_j}=a'\}-\frac{1}{c}\right)\\
=&\E \prod_{j \in J, s_j \ne s_0} \left(\bm1\{X_{s_j}=a\}-\frac{1}{c}\right)  \prod_{j \in J}\left(\bm1\{X_{s'_j}=a'\}-\frac{1}{c}\right)\E \left(\bm1\{X_{s_0}=a\}-\frac{1}{c}\right)=0.
\end{align*}
Similarly, the expectation vanishes if $s_0 \in \{s_j': j \in J\} \backslash \{s_j: j \in J\}$. 

Now, consider $\bm s, \bm s' \in V(G_n)_{|V(H)|}$ such that $\{s_j: j \in J\} = \{s_j': j \in J\}$. 
Define
\begin{align*}
W_J(a, a')&:=\E \prod_{j \in J} \left(\bm1\{X_{s_j}=a\}-\frac{1}{c}\right)\left(\bm1\{X_{s'_j}=a'\}-\frac{1}{c}\right). 
\end{align*} 
Note that $|W_J(a, a')| \leq 1$, for all $a, a' \in [c]$. Therefore, \eqref{eq:exp2QI}  gives
\begin{align}\label{eq:exp2QII}
\E T_J(H, G_n)^2 & = \sum_{a, a' \in [c]} \sum_{\substack{\bm s, \bm s' \in V(G_n)_{|V(H)|}\\ \{s_j: j \in J\} = \{s_j': j \in J\}}} M_{G_n}(\bm s, H) M_{G_n}(\bm s', H)  \frac{W_J(a, a')}{c^{2|V(H)|-2|J|}}   \nonumber \\
& \lesssim_{c} \sum_{\substack{\bm s, \bm s' \in V(G_n)_{|V(H)|}\\ \{s_j: j \in J\} = \{s_j': j \in J\}}} M_{G_n}(\bm s, H) M_{G_n}(\bm s', H)  \nonumber \\
& = O(|V(G_n)|^{2|V(H)|-|J|}) =o(|V(G_n)|^{2|V(H)|-2}),
\end{align}
whenever $|J| \geq 3$. 

Combining \eqref{eq:expQ} and \eqref{eq:exp2QII} it follows that $T_J(H, G_n)=o_P(|V(G_n)|^{|V(H)|-1})$, whenever $|J| \geq 3$.
\end{proof}

\begin{defn}\label{defn:Mnuvij} Let $H$ be a labeled finite simple graph. Then, for $1 \leq u \ne v \leq |V(H)|$ and $1 \leq i \ne j \leq |V(G_n)|$, define $M_{u, v}(i, j, H, G_n)$ as the number of injective homomorphism $\phi: V(H)\rightarrow V(G_n)$ such that $\phi(u)=i$ and $\phi(v)=j$. More formally,
\begin{align}
& M_{u, v}(i, j, H, G_n) \nonumber \\
&=a^+_{ij,uv}(G_n) \sum_{\substack{\bm s\backslash \{s_u, s_v\} \\ \bm s \in V(G_n)_{|V(H)|} }}\prod_{x \in N_H(u)\backslash \{v\}}a_{i s_x}(G_n) \prod_{y \in N_H(v) \backslash \{u\}} a_{j s_y}(G_n) \prod_{(x, y) \in E(H\backslash \{u, v\}) } a_{s_x s_y}(G_n), \nonumber 
\end{align}
with $a^+_{ij,uv}(G_n)=a_{ij}(G_n)$ if $(u, v) \in E(H)$ and 1 otherwise, and the sum is over indices $\bm s\backslash \{s_u, s_v\}$, with $\bm s \in V(G_n)_{|V(H)|}$, which are distinct and belong to $[|V(G_n)|]\backslash \{i, j\}$. Note that $M_{u, v}(i, j, H, G_n)$ is, in general, not symmetric in $i, j$, but satisfies $M_{v, u}(i, j, H, G_n)=M_{u, v}(j, i, H, G_n)$.\footnote{
For example, when $H=K_{1, 2}$ is the 2-star with the central vertex labeled 1, then 
$M_{1, 2}(i, j, K_{1, 2}, G_n)=M_{1, 3}(i, j, K_{1, 2}, G_n) = a_{ij}(G_n) (d_{G_n}(i)-a_{ij}(G_n))$, where $d_{G_n}(i)$ is the degree of the vertex $i$ in $G_n$, and $M_{2, 3}(i, j, K_{1, 2}, G_n)= \sum_{k \ne \{i, j \}} a_{ik}(G_n) a_{jk}(G_n)$, the number of common neighbors of $i, j$.} Finally, define the symmetric {\it scaled $2$-point homomorphism matrix} as $((\overline {\bm B}_H(G_n)_{ij}))_{i, j \in [|V(G_n)|] }$ with 
\begin{align}\label{eq:BH}
\overline {\bm B}_H(G_n)_{ij} &:=\frac{1}{2 |Aut(H)| \cdot |V(G_n)|^{|V(H)|-1}} \cdot \sum_{ 1 \leq u \ne v \leq |V(H)| } M_{u, v}(i, j, H, G_n), 
\end{align}
for $1 \leq i \ne j \leq |V(G_n)|$.
\end{defn}

The following lemma shows that $\Gamma(H, G_n)$  is a sum of $c$ quadratic forms in terms of the  scaled $2$-point homomorphism matrix, up to $o_P(1)$ terms.

\begin{lem}\label{lm:delta2} Recall $\Gamma(H, G_n)$ from \eqref{eq:gamma} and define, 
\begin{align*}
\Gamma_2(H, G_n) & := \frac{1}{c^{|V(H)|-2}}  \sum_{a=1}^c \sum_{1\le i\ne j\le |V(G_n)|} \overline{\bm B}_H(G_n)_{ij}  \left(\bm1\{X_{i}=a\}-\frac{1}{c}\right)\left(\bm1\{X_{j}=a\}-\frac{1}{c}\right), 
\end{align*} 
where $\overline{\bm B}_H(G_n)$ is the 2-point homomorphism matrix as defined in \eqref{eq:BH}. Then, $\Gamma(H,G_n)=\Gamma_2(H,G_n)+o_P(1)$.
\end{lem}

\begin{proof} 
Recalling \eqref{eq:TJHGn}, note that 
\begin{align*}
&c^{|V(H)|-2} \sum_{a=1}^c \sum_{\substack{J \subseteq V(H) \\ |J| = 2}}  T_J(H, G_n) \\
&=\sum_{a=1}^c  \sum_{1 \leq u < v \leq |V(H)|} \sum_{1 \leq s_{u} \ne s_{v} \leq |V(G_n)|} \frac{M_{u, v}(s_{u}, s_{v}, H, G_n)}{|Aut(H)|}  \left(\bm1\{X_{s_u}=a\}-\frac{1}{c}\right)\left(\bm1\{X_{s_v}=a\}-\frac{1}{c}\right) , \\
& = \sum_{a=1}^c  \sum_{1 \leq u < v \leq |V(H)|} \sum_{1 \leq i \ne j \leq |V(G_n)|} \frac{M_{u, v}(i, j, H, G_n)}{|Aut(H)|}   \left(\bm1\{X_{i}=a\}-\frac{1}{c}\right)\left(\bm1\{X_{j}=a\}-\frac{1}{c}\right), \\
&= |V(G_n)|^{|V(H)|-1} \cdot \sum_{a=1}^c  \sum_{1\le i\ne j\le |V(G_n)|} \bm B_H(G_n)_{ij}  \left(\bm1\{X_{i}=a\}-\frac{1}{c}\right)\left(\bm1\{X_{j}=a\}-\frac{1}{c}\right),
\end{align*} 
where $\bm B_H(G_n)=((\bm B_H(G_n)_{ij}))_{i, j \in [|V(G_n)|]}$ is a matrix with 
\begin{align*}
\bm B_H(G_n)_{ij}&=\frac{1}{|V(G_n)|^{|V(H)|-1}} \cdot \frac{1}{|Aut(H)|} \sum_{ 1 \leq u < v \leq |V(H)| } M_{u, v}(i, j, H, G_n). 
\end{align*}
Now, from \eqref{eq:BH} it is easy to see that $\overline {\bm B}_H(G_n)_{ij}=\frac{\bm B_H(G_n)_{ij}+\bm B_H(G_n)_{ji}}{2}$,  
which along with Lemma \ref{lm:J2} gives the desired conclusion.
\end{proof}

Next, define the analogous random variable for $\Gamma_2(H,G_n)$, where the centered color vectors $\{\bm R_v : v \in V(G_n) \}$, where $\bm R_v=(\bm 1\{X_v=a\}-\frac{1}{c})_{a \in [c]}$, are replaced by a collection of i.i.d. Gaussian vectors with the same mean and covariance structure. More formally, 
\begin{align}\label{eq:Q}
Q_2(H, G_n)& :=\frac{1}{c^{|V(H)|-2}}  \sum_{a=1}^c \sum_{1 \leq i \ne j \leq |V(G_n)|} \overline{\bm B}_H(G_n)_{ij}  \hat U_{i, a}\hat U_{j, a}.
\end{align} 
with $\hat U_{v, a}=U_{v, a}-\overline{U}_{v.}$, where $\{U_{v, a}:v\in V(G_n), a\in [c]\}$ are i.i.d. Gaussians with mean 0 and variance $1/c$ random variables and $\overline{U}_{v.}=\frac{1}{c}\sum_{a=1}^c U_{v, a}$. Note that for each $v\in V(G_n)$ the random vector $\hat{\bm U}_v:=(\hat U_{v, 1}, \hat U_{v, 2}, \ldots, \hat U_{v, c})$ has mean 0 and the same covariance matrix as $\bm R_v$. 
Also, $\{\hat{\bm U}_v, v\in V(G_n)\}$ are independent and identically distributed random vectors. Finally, define 
\begin{equation}\label{eq:delta}
\Delta_2(H, G_n):=\frac{Q_2(H, G_n)}{|V(G_n)|^{|V(H)|-1}}.
\end{equation}

The next lemma shows that the moments of $\Gamma_2(H, G_n)$ and $\Delta_2(H, G_n)$ are asymptotically close.

\begin{lem}\label{lm:invariance} 
For every integer $r \geq 1$, 
\begin{align}\label{eq:invariance}
\lim_{n \rightarrow \infty}\Big\{\E\Gamma_2(H, G_n)^r-\E\Delta_2(H, G_n)^r\Big\}=0.
\end{align}
\end{lem}

\begin{proof} Fix an integer $r\geq 1$. Using the bound $|V(G_n) |\max_{i,j\in V(G_n)} \overline{\bm B}_H(G_n)_{i j }\lesssim_H  1$, a direct expansion gives 
\begin{align*} 
& c^{r|V(H)|-2r}\left|\E \Gamma_2(H,G_n)^r-\E \Delta_2(H,G_n)^r\right| \nonumber \\
\lesssim_H &\frac{1}{|V(G_n)|^r}\sum_{a_1,\cdots, a_r \in [c]} \sum_{\substack{1 \leq i_1 \ne j_1 \leq |V(G_n)| \\ \vdots \\ 1 \leq i_r \ne j_r  \leq |V(G_n)|}} \left|\E \prod_{s=1}^r \left(\bm1\{X_{i_s}=a_s\}-\frac{1}{c}\right)\left(\bm1\{X_{j_s}=a_s\}-\frac{1}{c}\right)-\E\prod_{s=1}^r\hat U_{i_s,a_s} \hat U_{j_s,a_s}\right|.
\end{align*}

Now, fix an index set $J:=\{i_1,\cdots,i_r,  j_1,\cdots,j_r\}$, where $1 \leq i_s \ne j_s \leq |V(G_n)|$, for $s \in [r]$. If an index in $J$ appears exactly once, then using $\E \left(\bm1\{X_{i}=a\}-\frac{1}{c}\right)=\E \hat  U_{i, a}=0$, for every $i \in V(G_n)$ and $a \in [c]$, it is easy to see that both the moments inside the absolute value vanish. Therefore, we can assume that every index in $J$ appears at least twice. Moreover, as the total number of terms with at most $r-1$ distinct indices from $J$ is bounded above by $|V(G_n)|^{r-1}$, it suffices to consider the terms where the number of distinct indices from $J$ is exactly $r$, up to a $o(1)$-term. But in this case every index in $J$ appears exactly twice, and to prove \eqref{eq:invariance} it suffices to show  for any such index set $(i_1,j_1),\cdots, (i_r,j_r)$ we have $$\left|\E \prod_{s=1}^r \left(\bm1\{X_{i_s}=a_s\}-\frac{1}{c}\right)\left(\bm1\{X_{j_s}=a_s\}-\frac{1}{c}\right)-\E\prod_{s=1}^r\hat U_{i_s,a_s} \hat U_{j_s,a_s}\right|=0.$$
Indeed, in this case both the moments factorize over the distinct indices, and to show equality of moments it suffices to check that for all $a,b\in [c]$ we have 
$$\E  \left(\bm1\{X_{i}=a\}-\frac{1}{c}\right)\left(\bm1\{X_{i}=b\}-\frac{1}{c}\right)=\E\hat U_{i,a} \hat U_{i,b}.$$
This follows on noting that both sides equal $\frac{1}{c}\Big(1-\frac{1}{c}\Big)$ if $a=b$, and $-\frac{1}{c^2}$ otherwise. 
\end{proof}

From Lemmas \ref{lm:delta2} and \ref{lm:invariance}, to derive the limiting distribution of $\Gamma(H, G_n)$ it suffices to derive the limiting distribution of $\Delta_2(H, G_n)$, which is the sum of $c$-quadratic forms in $\overline{\bm B}_H(G_n)$. To this end, we need to understand the spectrum of the matrix $\overline{\bm B}_H(G_n)$. We  begin by defining the notion of cycles formed by $H$, which arise in the analysis of the power-sum of the eigenvalues of $\overline{\bm B}_H(G_n)$. 

\begin{defn}\label{defn:cycleH}
Fix an integer $g \geq 2$, and let $H_1, H_2, \ldots, H_g$ be $g$ isomorphic copies of $H$, where the image of the vertex $z \in V(H)$ in $H_a$ will be denoted by $z^{(a)}$, for $a \in [g]$. Then fixing indices $J:=\{(u_a, v_a): 1 \leq u_a \ne v_a \leq |V(H)|, a \in [g]\}$, define the {\it $r$-cycle of $H$ with pivots at $J$} as the graph obtained by the union of $H_1, H_2, \ldots, H_g$, where the vertex $v_a^{(a)} \in V(H_a)$ identified with the vertex $u_{a+1}^{(a+1)} \in V(H_{a+1})$, for $a \in [g]$, with $u_{g+1}^{(g+1)}:=u_1^{(1)}$ and $H_{g+1}=H_1$. Denote this graph by $H^{(g)}(J)$.  From Definition \ref{defn:tabxyHW}, it is easy to see that 
\begin{align}\label{eq:tHgW}
t(H^{(g)}(J), W)= \int_{[0, 1]^g}\prod_{a=1}^{g}  t_{u_a, v_a}(x_a, x_{a+1}, H, W) \prod_{a=1}^g \mathrm d x_{a}.
\end{align}
Figure \ref{fig:Hcycle} shows a 5-cycle of $K_{1,2}$ and a 6-cycle of $C_4$, and the associated pivots. 
\end{defn}

\begin{figure*}[h]
\centering
\begin{minipage}[c]{1.0\textwidth}
\centering
\includegraphics[width=4.9in]
    {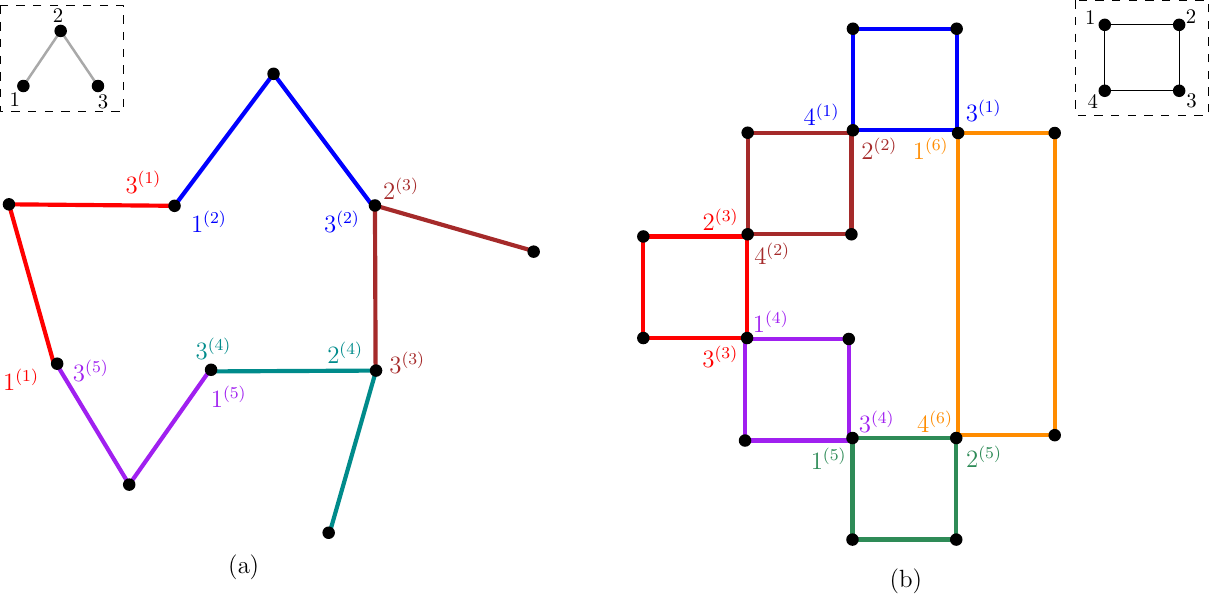}\\
\end{minipage} 
\caption{\small{A 5-cycle of $K_{1,2}$ with pivots $\{(1, 3), (1, 3), (2, 3), (2, 3), (1, 3)\}$, (b) a 6-cycle of $C_4$ with pivots $\{(3, 4), (2, 4), (2, 3), (1, 3), (1, 2), (4, 1)\}$.}}
\label{fig:Hcycle}
\end{figure*}

Equipped with the above definitions and recalling the function $W_H$ from \eqref{eq:WH}, we proceed to prove the convergence of the spectrum of $\overline{\bm B}_H(G_n)$.

\begin{lem}\label{lm:eigenpower} Let $\{\lambda_1(\overline{\bm B}_H(G_n)), \lambda_2(\overline{\bm B}_H(G_n)), \cdots, \lambda_{|V(G_n)|}(\overline{\bm B}_H(G_n))\}$ be the multi-set of eigenvalues of $\overline{\bm B}_H(G_n)$. Then, for every $g \geq 2$, $\lim_{n \rightarrow \infty}\sum_{r=1}^{|V(G_n)|} \lambda_r(\overline{\bm B}_H(G_n))^g=\sum_{i=1}^\infty \lambda_r(H,W)^g$, where $\{\lambda_1(H,W), \lambda_2(H,W), \ldots, \}$ is the multi-set of eigenvalues of $W_H$.  
Moreover, the assumption $t(H, W)>0$ ensures that the spectrum of $W_H$ is non-trivial, that is, $W_H$ has at least one (non-zero) eigenvalue. 
\end{lem}

\begin{proof} Fix $g \geq 2$. Define $K_g:=\frac{1}{2^g|Aut(H)|^g}$. Then 
\begin{align*}
& \sum_{r=1}^{|V(G_n)|} \lambda_r(\overline{\bm B}_H(G_n))^g \\
&=\tr(\overline{\bm B}_H(G_n)^g) \\ 
&=\sum_{\bm j \in V(G_n)^g} \prod_{a=1}^g \overline B_H(G_n)_{j_a j_{a+1}} \nonumber \tag*{(where $j_{g+1}=j_1)$}\\ 
&=K_g \cdot \frac{1}{|V(G_n)|^{g|V(H)|-g}}\sum_{\bm j \in V(G_n)^g} \prod_{a=1}^{g}\sum_{1 \leq u \ne v \leq |V(H)|} M_{u, v}(j_a, j_{a+1}, H, G_n) \nonumber \tag*{(recall \eqref{eq:BH})} \\ 
 &=K_g \cdot \frac{1}{|V(G_n)|^{g |V(H)|-g}} \sum_{1 \leq u_1 \ne v_1 \leq |V(H)|} \cdots \sum_{1 \leq u_g \ne v_g  \leq |V(H)|} \sum_{\bm j \in V(G_n)^g} \prod_{a=1}^{g} M_{u_a, v_a}(j_a, j_{a+1}, H, G_n). 
\end{align*}

Now, fix $J=\{(u_a, v_a): 1 \leq u_a \ne v_a \leq |V(H)|, a \in [g]\}$. Recalling Definition \ref{defn:Mnuvij}, note that $M_{u_a, v_a}(j_a, j_{a+1}, H, G_n)$ is the sum over $\bm s\backslash\{s_{u_a}, s_{v_a}\}$, where $\bm s \in V(G_n)_{|V(H)|}$, that is, the sum ranges over $|V(H)|-2$ indices taking values in $\{1, 2, \ldots, |V(G_n)|\}$. This implies, the product 
\begin{align}\label{eq:prodHGn}
\sum_{\bm j \in V(G_n)^g} \prod_{a=1}^{g} M_{u_a, v_a}(j_a, j_{a+1}, H, G_n),
\end{align} 
can be expanded to obtain a sum over at most $g+g(|V(H)|-2)=g |V(H)|-g$ indices taking values in $\{1, 2, \ldots, |V(G_n)|\}$.  However, if any two of the indices are the same, then the corresponding term in the sum is $o(|V(G_n)|^{g |V(H)|-g})$. (Here, we bound each $a_{ij}(G_n)$ by 1 and the sum over each distinct index by $|V(G_n)|$.) This shows that the leading term in \eqref{eq:prodHGn} is a sum over $g |V(H)|-g$ distinct indices ranging in $V(G_n)$, which counts the number of injective homomorphisms of $H^{(g)}(J)$ in $G_n$, where $H^{(g)}(J)$ is the $g$-cycle of $H$ with pivots at $J$, as in Definition \ref{defn:cycleH}.  Therefore, 
\begin{align*}
\frac{1}{|V(G_n)|^{g |V(H)|-g}} \sum_{\bm j \in V(G_n)^g}  \prod_{a=1}^{g} M_{u_a, v_a}(j_a, j_{a+1}, H, G_n) =t(H^{(g)}(J),G_n)+o(1) \rightarrow   t(H^{(g)}(J), W).
\end{align*}

Next, recall that $\lambda_1(H,W),\lambda_2(H,W), \ldots$ are the eigenvalues of the function $W_H$, as defined in Theorem \ref{thm:cfixed}. Then by the spectral theorem \cite[Section 7.5]{lovasz_book}, 
\begin{align}
\sum_{r=1}^\infty\lambda_r^g(H,W)&= t(C_g, W_H) \nonumber \\
&=\int_{[0, 1]^g} \prod_{a=1}^g W_H(x_a, x_{a+1})  \prod_{a=1}^g \mathrm dx_a \tag*{($x_{g+1}:=x_1$)}\nonumber \\
&=K_g \int_{[0, 1]^g} \prod_{a=1}^g \sum_{1\leq u \ne v \leq |V(H)|} t_{u,v}(x_a,x_{a+1},H,W)
 \prod_{a=1}^g \mathrm dx_a \tag*{(by \eqref{eq:WH})}\nonumber \\
&=K_g \sum_{1 \leq u_1 \ne v_1 \leq |V(H)|} \cdots \sum_{1 \leq u_g \ne v_g  \leq |V(H)|}  \int_{[0, 1]^g}\prod_{a=1}^{g}  t_{u_a, v_a}(x_a, x_{a+1}, H, W) \prod_{a=1}^g \mathrm d x_{a} \nonumber \\ 
\label{eq:lambdaHW} &=K_g \sum_{1 \leq u_1 \ne v_1 \leq |V(H)|} \cdots \sum_{1 \leq u_g \ne v_g  \leq |V(H)|} t(H^{(g)}(J),W), 
\end{align}
where the last step uses \eqref{eq:tHgW}. Therefore, \eqref{eq:lambdaHW} implies that the $g$-th power sum of the eigenvalues of $\overline{\bm B}_H(G_n)$ converge to the $g$-th power sum of eigenvalues of $W_H$, for every $g \geq 2$.

Finally, note that for $g=2$ and any set of pivots of the form $J=\{(a, b), (b, a) \}$, where $1 \leq a \ne b \leq |V(H)|$, $H^{(2)}(J)=H^2_{(a, b)}$ (recall Definition \ref{defn:H2}) and by Lemma \ref{lm:count}, $t(H^{(g)}(J),W) > 0$. Therefore, $\sum_{r=1}^\infty\lambda_r^2(H,W) >0$, which implies that the spectrum of $W_H$ is non-trivial. 
\end{proof}

Having established the convergence of the spectrum of $\overline{\bm B}_{H}(G_n)$, it remains to derive the asymptotic distribution of $\Delta_2(H, G_n)$, and hence $\Gamma(H, G_n)$.  This follows by the lemma below, which can be easily proved by computing the moment generating function of $\Delta_2(H, G_n)$ using the spectral decomposition, as in \cite[Lemma 7.3]{BDM}.

\begin{lem}\cite[Lemma 7.3]{BDM} \label{lm:previous_paper}
Let $\bm Q_n=((Q_n(i, j)))_{i, j \in [|V(G_n|]}$ be a sequence of symmetric $|V(G_n)|\times |V(G_n)|$ matrices with zeros on the diagonal. If there exists constants $\lambda_1, \lambda_2, \ldots$ such that $\lim_{n \rightarrow \infty}\tr(\bm Q_n^s)=\sum_{r=1}^\infty \lambda_r^s< \infty$, for every $s\ge 2$, then $$\sum_{a=1}^c\sum_{1\le i\ne j\le |V(G_n)|}Q_n(i, j) \hat U_{i, a} \hat U_{j, a} \dto \frac{1}{c} \sum_{r=1}^\infty \lambda_r\eta_r, $$
where $\{\eta_r\}_{r \geq 1}$ is a collection of i.i.d. $\chi_{(c-1)}^2-(c-1)$ random variables. \qed
\end{lem}

 \subsection{Examples} 
\label{sec:examplecfixed}

To begin with, we consider monochromatic edges, that is, $H=K_2$. In this case, the 2-point homomorphism matrix is just the scaled adjacency matrix of $G_n$, and we re-derive \cite[Theorem 1.4]{BDM}.

\begin{example}(Monochromatic Edges) Let $G_n$ be a sequence of graphs converging to the graphon $W$ and $H=K_2$. Then $|Aut(H)|=2$ and $W_{K_2}(x, y)= \frac{1}{2} W(x, y)$, and $\lambda_r(W_{K_2}) = \frac{1}{2} \lambda_r(W)$, where $\lambda_1(W), \lambda_2(W), \ldots $ are the eigenvalues of the operator $T_W: L_2[0, 1]\rightarrow L_2[0, 1]$, defined as $(T_Wf)(x)=\int_0^1W(x, y)f(y)\mathrm dy$. Then Theorem \ref{thm:cfixed} shows 
$$\Gamma(K_2, G_n) \dto \frac{1}{2 c}\sum_{r=1}^\infty \lambda_r(W) \eta_r,$$
where $\{\eta_r\}_{r \in \N}$ are independent $\chi_{(c-1)}^2-(c-1)$ random variables, as in  \cite[Theorem 1.4]{BDM}.
\end{example}

As before, Theorem \ref{thm:cfixed} applies to convergent sequence of dense random graphs, when the limit in \eqref{eq:graph_limit} hold in probability.

\begin{example}(Erd\H os-R\'enyi random graph) Let $G_n \sim G(n, p)$ be the Erd\H os-R\'enyi random graph and $H$ be any finite simple graph. In this case, $G_n$ converges to the constant graphon $W^{(p)}=p$, the constant function $p$, and, from \eqref{eq:WH}, $W_{H}^{(p)}(x, y) =\frac{{|V(H)| \choose 2}}{|Aut(H)|} p^{|E(H)|}$.  It is easy to see that $W^{(p)}_{H}$ has only 1 non-zero eigenvalue $\lambda_1(W^{(p)}_H)= \frac{{|V(H)| \choose 2}}{|Aut(H)|} p^{|E(H)|}$. Therefore, by Theorem \ref{thm:cfixed}, 
\begin{align}\label{eq:gammaKn}
\Gamma(H, K_n) \dto  \frac{\sigma_{H, p}}{c^{|V(H)|-1}} \cdot \left(\chi^2_{(c-1)}-(c-1) \right).  
\end{align}
where $\sigma_{H, p}:= \frac{{|V(H)| \choose 2}}{|Aut(H)|}p^{|E(H)|}$. 
\end{example}

As another example, consider the limiting distribution in a non-symmetric example: number of monochromatic 2-stars in a complete bipartite graph. 

\begin{example}Let $G_n=K_{\ceil{\frac{n}{2}}, \ceil{\frac{n}{2}}}$ and $H=K_{1, 2}$. Then $|Aut(H)|=2$, and $G_n$ converges to the graphon $W= \bm 1\{(x-\frac{1}{2})(y-\frac{1}{2}) \leq 0\}$. This implies $d_W(x)=\frac{1}{2}$ for all $x \in [0, 1]$ and
\begin{align*}
W_{K_{1, 2}}(x, y)&=\frac{ W(x, y) d_W(x) + W(x, y) d_W(y) +\int_{[0, 1]} W(x, z_1) W(y, z_1)  \mathrm dz_1}{2} \\
&=
\left\{
\begin{array}{ccc}
\frac{1}{2} & \text{if}  & (x-\frac{1}{2})(y-\frac{1}{2}) \leq 0  \\
\frac{1}{4} & \text{if}  & (x-\frac{1}{2})(y-\frac{1}{2}) > 0.   
\end{array}
\right.
\end{align*}
This function has two non-zero eigenvalues $\frac{3}{8}$ and $-\frac{1}{8}$ and by Theorem \ref{thm:cfixed}
$$\Gamma(K_{1, 2}, K_{\ceil{\frac{n}{2}}, \ceil{\frac{n}{2}}})\dto \frac{1}{8 c^2}\left(3 \eta_1- \eta_2 \right),$$
where $\eta_1$ and $\eta_2$ are independent $\chi^2_{(c-1)}-(c-1)$ random variables.
\end{example}

As a final example of Theorem \ref{thm:cfixed}, consider the limiting distribution of the number of monochromatic triangles in a complete tripartite graph.

\begin{example}
Let $G_n=K_{\ceil{\frac{n}{3}}, \ceil{\frac{n}{3}},\ceil{\frac{n}{3}}}$ and $H=K_{3}$. Then $|Aut(H)|=3$, and $G_n$ converges to the graphon $W=\bm 1\{(x,y)\in \sS^c\}$, where $\sS:=[0,\frac{1}{3}]^2\bigcup [\frac{1}{3},\frac{2}{3}]^2\bigcup [\frac{2}{3},1]^2.$
A direct computation gives that for all $(u,v)\in V(K_3)$, with $u\ne v$,  
$$t_{u,v}(x,y,K_3,W)=W(x,y)\int_0^1 W(x,z)W(y,z)\mathrm dz=\frac{1}{3} \bm 1\{(x,y)\in \sS^c\},$$ which implies $W_{K_{3}}(x, y)=\frac{1}{3}\bm1\{(x,y)\in \sS^c\}$. Now, since $W_{K_3}$ has eigenvalues $\frac{2}{9},-\frac{1}{9},-\frac{1}{9}$, Theorem \ref{thm:cfixed} gives
$$\Gamma(K_{3}, K_{\ceil{\frac{n}{3}}, \ceil{\frac{n}{3}},\ceil{\frac{n}{3}}})\dto \frac{1}{9 c^2}\left(2 \eta_1- \eta_2 -\eta_3\right),$$
where $\eta_1,\eta_2,\eta_3$ are independent $\chi^2_{(c-1)}-(c-1)$ random variables.
\end{example}

We conclude with an example, which shows, as before, that the condition $t(H,W)>0$ is necessary for $\Gamma(H,G_n)$ to have a non-degenerate limit as an infinite sum of chi-squared random variables.

\begin{example}
Let $G_n=K_{1,n,n}$ be the complete 3-partite graph,  with partitions $\{z\},B,C$, and $H=K_3$. Given the color of the vertex $z$ is $a$, using the same notations as in Example \ref{ex:productpoisson}, both $L_n(a)$ and $R_n(a)$ are independent $\dBin(n,1/c)$, and consequently 
$$\sqrt{n}\Gamma(H,G_n)=\frac{T(K_3,G_n)-\E T(K_3,G_n)}{n^{\frac{3}{2}}}=\frac{L_n(a)R_n(a)-\frac{n^2}{c^2}}{n^{\frac{3}{2}}}\dto  N\left(0,\frac{2}{c}\left(1-\frac{1}{c}\right)\right).$$
Therefore, unconditionally $\sqrt{n}\Gamma(H,G_n)$ converges to a Gaussian as well, which cannot be expressed as an infinite sum of chi-squared random variables.
 \end{example}

\begin{remark} A similar thing happens in Example \ref{ex:normaldiamond}, where $G_n$ is the disjoint union of the $n$-pyramid $\cD_n$ and the complete bipartite graph $K_{n, n}$ and $H=K_3$ is the triangle. In this case, it is easy to see that $T(K_3, G_n)=\frac{1}{c}  \dBin(n, \frac{1}{c})  + (1-\frac{1}{c}) \delta_0$, a mixture of a $\dBin(n, \frac{1}{c})$ and a point mass at zero. This implies, $\Gamma(H, G_n)$ does not have a non-degenerate limiting distribution. \\ 
\end{remark}

\small

\small{\noindent{\bf Acknowledgement:} The authors thank an anonymous referee for providing many careful comments, which greatly improved the quality and presentation of the paper.}

\end{document}